\documentclass[a4paper,10pt]{amsart}

\usepackage{graphicx, psfrag}
\usepackage{epstopdf}	

\usepackage{amsmath, amssymb, amscd, amsfonts, mathrsfs, epsfig, epic, verbatim}
\usepackage[all]{xy}
\usepackage[mathcal]{euscript}
\usepackage{epsfig}
\usepackage{times}
\usepackage{hyperref}
\usepackage{MnSymbol}
\usepackage{color}
\hypersetup{pdfborder={0 0 0}}
\hypersetup{pdfencoding=auto}			
\usepackage{enumitem}

 \newcommand{\Z}{{\mathbb Z}}                   
 \newcommand{\R}{{\mathbb R}}                   
 \renewcommand{\C}{{\mathbb C}}                   
 \newcommand{\CP}[1]{\mathbb{C}P^{#1}}          
 \newcommand{\Xt}{\widetilde{X}}			
 \newcommand{\Xtt}{Y}									
 \newcommand{\St}{\widetilde{\Sigma}}		
 \newcommand{\E}{{\mathcal E}}                  
 \newcommand{\Et}[1]{\widetilde{E}_{#1}}		
 \newcommand{\F}{{\mathcal F}}			
 \newcommand{\Qt}{{\mathcal Q}}			
 \newcommand{\Ft}{{\mathcal S}}			
 \renewcommand{\L}{{\mathcal L}}			
 \renewcommand{\P}{{\mathtt P}}                  
 \newcommand{\cJ}{\overline{\Pic}}			
 \renewcommand{\t}{\theta}                      
 \newcommand{\PP}{{\mathbf P}}                    
 \renewcommand{\O}{{\mathcal O}}                 
 \newcommand{\Mod}{{\mathcal M}}               
 \renewcommand{\d}{\mbox{d}}                      
 \newcommand{\gl}{\mathfrak{gl}}                
 \newcommand{\End}{{\mathcal E}nd}
 \newcommand{\Tor}{{\mathcal T}or}

 \DeclareMathOperator{\coker}{coker}
 \DeclareMathOperator{\im}{im}

 \DeclareMathOperator{\Pic}{Pic}
 
 \DeclareMathOperator{\ad}{ad}
 \DeclareMathOperator{\red}{red}

 \DeclareMathOperator{\pardeg}{par-deg}

\DeclareTextCommand{\textoneinferior}{PU}{\9040\201}				
\DeclareTextCommand{\texttwoinferior}{PU}{\9040\202}				
\DeclareTextCommand{\textseveninferior}{PU}{\9040\207}			
\DeclareTextCommand{\texteightinferior}{PU}{\9040\210}			

 \newtheorem{prop}{Proposition}[section]
 
 \newtheorem{rk}[prop]{Remark}
 
 \newtheorem{lem}[prop]{Lemma}
 \newtheorem{defn}[prop]{Definition}
 
 \newtheorem{thm}[prop]{Theorem}

 \title[Two-dimensional moduli spaces of irregular Higgs bundles]{Two-dimensional moduli spaces of rank $2$ Higgs bundles over $\CP{1}$ with one irregular singular point}
\author{P\'eter Ivanics}  
\address{Budapest University of Technology and Economics, 1111. Budapest,
Egry J\'ozsef utca 1. H \'ep\"ulet, Hungary}
\address{R\'enyi Institute of Mathematics, 1053. Budapest, Re\'altanoda
utca 13-15. Hungary}
\email{ipe@math.bme.hu}

\author{Andr\'as Stipsicz} 
\address{R\'enyi Institute of Mathematics, 1053. Budapest, Re\'altanoda
utca 13-15. Hungary}
\email{stipsicz@renyi.hu}

\author{Szil\'ard Szab\'o}
\thanks{Corresponding author: Szil\'ard Szab\'o}
\address{Budapest University of Technology and Economics, 1111. Budapest,
Egry J\'ozsef utca 1. H \'ep\"ulet, Hungary}
\address{R\'enyi Institute of Mathematics, 1053. Budapest, Re\'altanoda
utca 13-15. Hungary}
\email{szabosz@math.bme.hu}

\begin{document}
 \begin{abstract}
 We give a complete description of the two-dimensional moduli spaces of
 stable Higgs bundles of rank $2$ over $\CP{1}$ with one irregular
 singular point, having a regular leading-order term, and endowed with
 a generic compatible parabolic structure such that the parabolic
 degree of the Higgs bundle is $0$.  Our method relies on elliptic
 fibrations of the rational elliptic surface, an equivalence of
 categories between irregular Higgs bundles and some sheaves on a ruled
 surface, and an analysis of stability conditions.
 \end{abstract}

 \maketitle

 \section{Introduction}
 In this article we consider $2$ complex dimensional moduli spaces of
 singular Higgs bundles over $\CP1$ with irregular singularities.  It
 is known \cite{Biq-Boa} that if one fixes finitely many points on a
 curve $C$ and suitable polar parts for a Higgs bundle near those
 points, then one gets a holomorphic symplectic moduli space of Higgs
 bundles over $C$ with the given {irregular part and residues} at
 the singularities.  In some cases these spaces turn out to be of
 complex dimension $2$.  Our aim in this article is to give a complete
 description of the two-dimensional holomorphic symplectic moduli
 spaces of rank $2$ Higgs bundles over $\CP1$ having a unique pole of
 order $4$ as singularity, and regular leading-order term.  One needs
 to distinguish two cases, depending on whether the leading-order term
 is a regular semi-simple endomorphism (untwisted case), or has
 non-vanishing nilpotent part (twisted case).  As we will see, the
 corresponding fiber at infinity of the Hitchin fibration is $\Et7$ in
 the untwisted case and $\Et8$ in the twisted case.  The corresponding
 de Rham moduli spaces of irregular connections are related to the
 Painlev\'e II (untwisted case) and Painlev\'e I (twisted case)
 equations. The polar part of an irregular Higgs bundle depends on some
 complex parameters 
 \begin{equation}\label{eq:untwisted-parameters}
    a_{\pm}, b_{\pm}, c_{\pm}, \lambda_{\pm} \in \C ,\quad a_+ \neq a_- \tag{{U}}
 \end{equation}
 in the untwisted case (referred to as (U)) and 
 \begin{equation}\label{eq:twisted-parameters}
   b_{-8},\ldots ,b_{-3} \in \C ,\quad  b_{-7} \neq 0 \tag{{T}}
 \end{equation}
 in the twisted case {(referred to as (T))}, see Subsection \ref{subsec:spectral}.

 In the following statements we let $\Mod$ be a moduli space of rank
 $2$, parabolic degree $0$ stable parabolic irregular Higgs bundles
 over $\CP1$ with a unique pole of order $4$ with a regular
 leading-order term and fixed parameters
 (\ref{eq:untwisted-parameters}) or (\ref{eq:twisted-parameters}).
 For details {and definitions} see Subsection
 \ref{subsec:spectral}.  {If the parabolic structure is generic},
 the degree of the underlying vector bundle is necessarily equal to
 $-1$.  It is expected that moduli spaces $ \Mod^{ss}$ of semi-stable
 irregular Higgs bundles with fixed polar parts underlie completely
 integrable systems with Abelian varieties as generic fibers.  If
 $\dim_{\C}(\Mod ^{ss}) = 2$ this would then imply that $\Mod ^{ss}$ is an
 elliptic fibration over a curve.  For generic weights $\Mod
 ^{ss}=\Mod ^s$, where $\Mod ^s$ is the moduli space of stable
 irregular Higgs bundles. Our results below will confirm this
 expectation, with one singular fiber of type $\Et7$ (untwisted case)
 or $\Et8$ (twisted case).  On the other hand, there are several
 possibilities for the other singular fibers \cite{Miranda, Persson, SSS}.

 In \cite{Sz-bnr}, a general equivalence of categories between
 irregular Higgs bundles and some pure $1$-dimensional rank one
 sheaves on a ruled surface was shown to hold, assuming that the
 leading order term of the Higgs field is semi-simple.  We will use
 this equivalence to prove our first result, giving a complete
 description of these further singular fibers in the untwisted case in
 terms of the parameters of \eqref{eq:untwisted-parameters}.  (For the
 definition of various types of singular fibers see \cite{Kodaira} or
 Section~\ref{sec:elliptic-fibrations}.)

 \begin{thm}\label{thm:mainE7}
   Assume that the polar part of the Higgs bundle is untwisted.  Then
   {the moduli space $\Mod ^s$} is biregular to the complement of the
   fiber at infinity (of type $\Et7$) in an elliptic fibration of the
   rational elliptic surface such that the set of other singular
   fibers of the {Hitchin} fibration is:
  \begin{enumerate}
   \item a type $III$ fiber if $\Delta = 0$ and $\lambda_+ = 0$; \label{thm:mainE7+III}
   \item a type $II$ and an $I_1$ fiber if $\Delta = 0$ and $\lambda_+ \neq 0$; \label{thm:mainE7+II+I1}
   \item an $I_2$ and an $I_1$ fiber if $\Delta \neq 0$ and $\lambda_+ = 0$; \label{thm:mainE7+I2+I1}
   \item and three $I_1$ fibers otherwise, \label{thm:mainE7+I1+I1+I1}
  \end{enumerate}
	where 
	$\Delta = \left(\left(b_--b_+\right){}^2-4 \left(a_--a_+\right) \left(c_--c_+\right)\right){}^3-432 \left(a_--a_+\right){}^4 \lambda _+^2.$
 \end{thm}
\begin{rk}
Since by Equation~(\ref{eq:residue}) in Subsection~\ref{subsec:spectral}
we have that $\lambda _++\lambda _-=0$, the above conditions could be
phrased in terms of $\lambda _-$ as well.
\end{rk}
 Notice that according to \cite[Proposition~4.2]{SSS} 
this is a complete list of the possible singular fibers 
 of elliptic fibrations on the rational elliptic surface without multiple fibers and having a singular fiber of type $\Et7$. 
 The proof of Theorem \ref{thm:mainE7} is given in Sections \ref{sec:E7} and \ref{sec:stability-analysis}, 
 where an explicit description of the Hitchin fibers corresponding to the reducible singular curves in the fibration 
 is given. In Section \ref{sec:stability-analysis} we also work out the stability analysis in the case of rank $2$ irregular 
 Higgs bundles in the degree $0$ case; strictly speaking we do not need this analysis to prove the theorem, nevertheless 
 we found it interesting enough to include it. 

 Similarly to Theorem \ref{thm:mainE7}, the next theorem provides a complete description of the singular fibers of   
 the fibration in the twisted case, in terms of the parameters (\ref{eq:twisted-parameters}). 

 \begin{thm}\label{thm:mainE8}
   Assume that the polar part of the Higgs bundle is twisted.  Then
   the {moduli space $\Mod ^s$} is biregular to the complement of the
   fiber at infinity (of type $\Et8$) in an elliptic fibration of the
   rational elliptic surface such that the set of other singular
   fibers of the {Hitchin} fibration is:
  \begin{enumerate}
   \item a type $II$ fiber if $D=0$;\label{thm:mainE8+II}
   \item and two type $I_1$ fibers otherwise,\label{thm:mainE8+I1+I1}
  \end{enumerate}
	where $D=\left(b_{-6}^2+4 b_{-5}\right){}^2-24 b_{-7} \left(b_{-6} b_{-4}+2 b_{-3}\right)$.
 \end{thm}
 Notice again that according to \cite[Section~4.1]{SSS} this is a
 complete list of the possible singular fibers of elliptic fibrations
 without multiple fibers and having a singular fiber of type $\Et8$.
 We prove Theorem~\ref{thm:mainE8} in Section \ref{sec:E8}.

 Now let us give an outline of the paper.  In Section
 \ref{sec:prep-mat} we fix our notations and provide some well-known
 background material used later.  
 In Section \ref{sec:elliptic-fibrations} we give a detailed analysis
 of elliptic fibrations on the rational elliptic surface with one
 singular fiber of type $\Et7$ or of type $\Et8$.  In Section
 \ref{sec:E7} we first construct the rational surface $\Xtt$ governing
 the moduli space $\Mod$ in the untwisted case.  Quoting the general
 categorical equivalence of \cite{Sz-bnr}, we then achieve the proof
 of Theorem~\ref{thm:mainE7}, up to the stability analysis of
 irregular Higgs bundles with reducible spectral curve. This latter,
 in turn, is carried out in Section \ref{sec:stability-analysis}.  The
 analysis of the case of a type $I_2$ curve proceeds along the lines
 of Section 4 of Schaub's paper \cite{Sch}.

 We start Section \ref{sec:E8} by some straightforward computations
 expressing the coefficients of the Puiseux-expansion of the
 eigenvalues of the Higgs field in terms of the parameters
 (\ref{eq:twisted-parameters}).  We then go on to construct the
 rational surface $\Xtt$ governing the moduli space $\Mod$ in the
 twisted case.  Next, in Proposition \ref{prop:equivalence} we give an
 analogue of the general categorical equivalence of \cite{Sz-bnr}
 between twisted irregular Higgs bundles and some pure $1$-dimensional
 rank one sheaves on $\Xtt$.  This then allows us to prove Theorem
 \ref{thm:mainE8}.

 Let us make a few remarks on related literature. In the paper
 \cite{Sak}, spaces of initial conditions for Painlev\'e equations are
 studied using rational surfaces and root systems. In particular, in
 Appendix B {\it loc. cit.} configurations of curves similar to ours
 appear.  In \cite{GO} the singular fiber of the Hitchin map
 corresponding to a singular spectral curve of type $A_k$ is
 determined.  Our Section \ref{sec:stability-analysis} is reminiscent
 to (special cases of) their results.  The work \cite{GMN} (in
 particular, Section 9 thereof) undertakes the analysis of
 wall-crossing phenomena related to Hitchin systems with irregular
 singularities.  Finally, let us mention that we hope to treat the
 $2$-dimensional moduli spaces of rank $2$ irregular Higgs bundles
 over $\CP1$ with several marked points in the future, cf. \cite{ISSz2}.
 
\bigskip

\noindent {\bf {Acknowledgments}}: 
The third author was supported by NKFIH K120697. 
The authors were supported
by NKFIH KKP126683, and by the \emph{Lend\"ulet} program of the 
Hungarian Academy of Sciences. 
They also want to thank the referee
for many useful comments and suggestions.

 \section{Preparatory material}\label{sec:prep-mat}
 We denote by $\O$ and $K$ the sheaf of regular functions and the
 canonical sheaf respectively.  We identify holomorphic line bundles
 over $\CP1$ with their sheaves of sections.  We equally let $\O(1)$
 stand for the ample line bundle and for $n\in \Z$ set $K(n) = K
 \otimes \O(n)$.

 \subsection{The second Hirzebruch surface and the basic birational map}\label{subsec:Hirzebruch}
 Throughout the paper we will consider the  surface 
 $$
   X = \PP (K(4) \oplus \O ), 
 $$
 the fiberwise projectivization of the rank $2$ holomorphic line bundle $K(4) \oplus \O $ over $\CP{1}$.  
 Given that the line bundle $K(4)$ is isomorphic to $\O(2)$, we get that $X$ is biholomorphic to the Hirzebruch surface 
 of index $2$. The surface $X$ naturally fibers over $\CP{1}$ with fibers isomorphic to $\CP{1}$: 
\begin{equation}\label{eq:Hfib}
   p : X \to \CP{1}.
\end{equation}
 This morphism is sometimes called the ruling. 
 We denote its generic fiber by $F$ and the homology class of $F$ by $[F] \in H_2(X; \Z)$.

 It is known that $X$ admits two further remarkable closed curves denoted by
 $C_0, C_{\infty}$ and called the $0$-section and section at infinity,
 respectively.  Both $C_0$ and $C_{\infty}$ are sections of $p$, in particular
 they are biholomorphic to $\CP{1}$.  Specifically, if we let $\mathbf{0}$
 stand for the $0$-section of $K(4)$ and $\mathbf{1}$ stand for the constant
 section equal to $1$ of $\O$ then
 $$
   C_0 = \{ [\mathbf{0}_q:\mathbf{1}_q] \mid  q\in \CP{1} \}, 
 $$
 where the subscripts $q$ mean evaluation of the given sections at $q$, and as usual $[\cdot : \cdot ]$ denote projective 
 coordinates. 
 Locally, the section at infinity can be defined similarly, however it is not possible to pick a 
 single section of $K(4)$ because any such section vanishes at two points 
of $\CP{1}$. 
 So, letting $\kappa$ stand for a local non-vanishing section of $K(4)$ on some open set $U\subset \CP1$, 
 we define 
 $$
   C_{\infty}\cap p^{-1}(U) = \{ [\kappa (q):\mathbf{0}] \mid q\in U \}
 $$ where $\mathbf{0}$ stands for the $0$-section of $\O$.  It can be checked
   that if $V$ is another open subset of $\CP1$ with a non-vanishing section
   $\mu$ then these definitions of $C_{\infty}$ agree on $U \cap V$, hence
   these formulas give a well-defined curve.  We denote the homology classes
   defined by these sections by $[C_0], [C_{\infty}]$.

 The second homology $H_2(X;\Z)$ is generated by the classes of any two of the above three curves, 
 the relation between them being 
 $$
   [C_{\infty}] = [C_0] - 2 [F]. 
 $$
 The intersection pairing is given by the formulas 
 $$
   [C_{\infty}]^2 = -2, \quad [C_0]^2 = 2, \quad [F]^2 = 0, \quad [C_{\infty
   }]\cdot [C_0]=0, \quad [C_{\infty }]\cdot [F] = [C_0]\cdot [F]=1. 
 $$

 As it is well-known, $X$ is birational to $\CP2$ by the morphisms 
 \begin{equation}\label{eq:birational}
   \xymatrix{
     & \Xt \ar[ld] \ar[rd] & \\
   X \ar@{..>}[rr]^{\omega} && \CP2}
 \end{equation}
 where $\Xt \to X$ is the blow-up of a point $(\kappa(q) :
 \mathbf{1})\in X\setminus C_{\infty}$ for any $q\in U \subset \CP1$
 and local section $\kappa \in H^0(U;K(4))$, and $\Xt \to \CP2$ is the
 blow-up of two infinitely close points on $\CP2$.  For sake of
 concreteness, we may take the locus of this reduced point to be
 $(0:0:1)$.  The proper transform of the fiber $F_q$ of the map $p$ of
 Equation~\eqref{eq:Hfib} over $q\in \CP{1}$ is the exceptional
 divisor of the second blow-up of $\CP2$.  On the other hand, the
 proper (which in this case is the same as the total) transform of
 $C_{\infty}$ in $\Xt$ is equal to the proper transform of the
 exceptional divisor of the first blow-up of $\CP2$ under the second
 blow-up.  Throughout the paper we will use the above $\omega$ to go
 back and forth between $X$ and $\CP2$.

 \subsection{Elliptic fibrations and their relative compactified Picard schemes}\label{subsec:compactified-Jacobian}

 In this section we summarize some facts concerning families of curves
 that we will need in the paper.

 Let $B$ be a scheme over $\C$ and $X \to B$ be a flat projective map of relative dimension $1$. 
 For a geometric point $b$ of $B$ we call the fiber at $b$ the base change of $X$ under the inclusion map $b \to B$, 
 and we denote the fiber at $b$ by $X_b$. 
 Throughout this section we assume that for each geometric point $b$ of $B$ the fiber $X_b$ is reduced. 
 We furthermore assume that each singular fiber is of the following types: 
 \begin{enumerate}
  \item a simple nodal rational curve $I_1$; 
  \item two smooth rational curves meeting transversely in two distinct points $I_2$; 
  \item a cuspidal rational curve $II$.
 \end{enumerate}
(Again, for the definition of the various singularities appearing in
 elliptic fibrations see \cite{Kodaira} or
 Section~\ref{sec:elliptic-fibrations}. The case of type $III$
   singular fibers, also needed in the proof of
   Theorem~\ref{thm:mainE7}, will be discussed in
   Subsection~\ref{subsec:III}.)  In this situation there exists a
 relative compactified Picard scheme
 $$
   \cJ_{X|B}
 $$
 parametrizing torsion-free sheaves $\Ft$ of $\O_{X_b}$-modules of rank $1$. 
 It naturally decomposes according to the (total) degree $\delta$ of $\Ft$ as 
 \begin{equation}\label{eq:compactified-Picard-scheme-I1-II}
   \cJ_{X|B}^{\delta}
 \end{equation}
 where the degree is defined by 
 \begin{equation}\label{eq:degree-F}
    \deg (\Ft) = \chi (\Ft) - \chi (\O_{X_b})
 \end{equation}
 with $\chi$ standing for Poincar\'e characteristic.  For types $I_1$ and $II$
 the scheme $\cJ$ was constructed by \cite{DS}.  The $I_2$ case is a
 particular case of \cite{OS}; we will come back to this case in Subsection
 \ref{subsubsec:Oda-Seshadri-I2}.  In order to introduce the ideas to be used
 later in various other situations, let us give here the description of
 (\ref{eq:compactified-Picard-scheme-I1-II}) in the cases $I_1$ and $II$
 according to \cite[Section~13]{OS} and \cite[Chapter~4]{Cook}.  Our argument
 can be made more precise using generalized parabolic line bundles on the
 normalization introduced by \cite{Bh}.
 
 \begin{prop} \label{prop:compactified-Picard-scheme-I1-II}
 (Oda--Seshadri \cite{OS}, Altman--Kleiman \cite{AK})
 \begin{enumerate}
  \item  Let $X_b$ be a curve of type $I_1$. \label{prop:compactified-Picard-scheme-I1}
   Then for any $\delta \in \Z$ the scheme $\cJ_{X_b}^{\delta}$ is isomorphic to a curve of type $I_1$.
  \item  Let $X_b$ be a curve of type $II$. \label{prop:compactified-Picard-scheme-II}
   Then for any $\delta \in \Z$ the scheme $\cJ_{X_b}^{\delta}$ is isomorphic to a curve of type $II$.
 \end{enumerate}
 \end{prop}
 \begin{proof}
 We only treat part (\ref{prop:compactified-Picard-scheme-I1}). Let 
 $$
 \pi: \tilde{X}_b \to X_b
 $$
 stand for the normalization of $X_b$. Then $\tilde{X}_b$ is a smooth rational curve. 
 Let us denote by $x_0\in X_b$ the only singular point and by $0,\infty \in \tilde{X}_b$ its preimages under the map $\pi$. 
 Then a degree $0$ line bundle on $X_b$ is the same thing as a line bundle $L$ of degree $0$ on $\tilde{X}_b$ endowed with an 
 isomorphism 
 $$
   L_0 \cong L_{\infty}, 
 $$
 where $L_p$ denotes the fiber of $L$ over $p\in \tilde{X}_b$. 
 Now there is just one degree $0$ holomorphic line bundle on $\tilde{X}_b$, namely $L = \O_{\tilde{X}_b}$, 
 so the data above reduces to just the identification of the fibers. 
 This in turn can be described by the image $\lambda \in \C^{\times} \subset L_{\infty}$ of $1\in L_0$. 
 Intrinsically $\lambda$ can be understood as an element of the projective line 
 $$
   \PP (L_0 \oplus L_{\infty}). 
 $$
 Let us denote by $L(\lambda )$ the degree $0$ line bundle on $X_b$ obtained by the above identification of the fibers; 
 clearly, for $\lambda' \neq \lambda$ the line bundle $L(\lambda' )$ is not isomorphic to $L(\lambda )$. 
 To sum up, the universal line bundle on $X_b$ is given by 
 $$
   L(\cdot ) \to \C^{\times} \times X_b \subset \PP (L_0 \oplus L_{\infty})\times X_b . 
 $$
 Our aim is to find the limit of $L(\lambda )$ as $\lambda \to 0$ or $\infty$ in $\PP (L_0 \oplus L_{\infty})$. 
 In the case $\lambda = 0$ the limit consists of a line bundle on $\tilde{X}_b$ with an identification of the fiber 
 $L_0$ to $0 \in L_{\infty}$; said differently, there is a short exact sequence 
 $$
   0 \to L(0) \to \pi_* L \to L_0 \to 0,
 $$
 hence $L(0) = \pi_* \O_{\tilde{X}_b}(-\{ 0 \})$. 
 Similarly, the limit $\lambda \to \infty$ fits into the short exact sequence 
 $$
   0 \to L(\infty ) \to \pi_* L \to L_{\infty} \to 0, 
 $$
 hence $L(\infty ) = \pi_* \O_{\tilde{X}_b}(-\{ \infty \})$. 
 As $\tilde{X}_b$ is of genus $0$, the bundles $\O_{\tilde{X}_b}(-\{ 0 \})$ and $\O_{\tilde{X}_b}(-\{ \infty \})$ are 
 isomorphic to each other, therefore so are their direct images by $\pi$. 
 The statement in the case of $I_1$ now follows. 

 As for part (\ref{prop:compactified-Picard-scheme-II}), see 
 \cite[Theorem~18]{AK}.
 \end{proof}

 \subsubsection{Oda--Seshadri stability for $I_2$ curves}\label{subsubsec:Oda-Seshadri-I2}
 In this subsection we continue the summary of known results concerning compactified Picard schemes. 
 For families with singular fibers $I_n$ for $n\geq 2$ (and more generally, 
 for reduced curves with only simple nodes as singular points) 
 the compactifications of the Picard scheme were studied in \cite{OS}. 
 In this case, the degree of the restriction of $\Ft$ to each component of $X_b$ needs to be centered about some values. 
 Let us restrict our attention to the case $n=2$ and denote by $X_+,X_-$ the irreducible components of $X_b$. 
 These are smooth curves of genus $0$, attached at two points. 
 We may assume for ease of notations that the common points are 
 $0,\infty \in X_{\pm}$ so that $0\in X_+$ is identified with $0\in X_-$ and 
 $\infty\in X_+$ is identified with $\infty\in X_-$. 
 We will also denote by $0$ and $\infty$ the point of $X_b$ obtained by the above identification. 
 The curve 
 \begin{equation*}
  \tilde{X}_b = X_+ \coprod X_-
 \end{equation*}
 is called the normalization of $X_b$. There is an obvious map 
 $$
   \sigma: \tilde{X}_b \to X_b. 
 $$
 It turns out that in order to get a moduli scheme we need to impose a further condition of stability on 
 the sheaves $\Ft$ that we wish to parametrize. This stability condition depends on some parameters 
 $(\phi_+, \phi_- )\in\R^2$ satisfying
 $$	
   \phi_+ + \phi_- = 0. 
 $$ 
 For a torsion-free coherent sheaf $\Ft$ of $\O_{X_b}$-modules of rank $1$ let us set 
 \begin{equation}\label{eq:associated-invertible-sheaf}
  \L(\Ft) = \sigma^* \Ft/\Tor^{\O_{\tilde{X}_b}}(\sigma^* \Ft)
 \end{equation}
 with $\Tor^{\O_{\tilde{X}_b}}(\sigma^* \Ft)$ denoting the torsion part of the $\O_{\tilde{X}_b}$-module $\sigma^* \Ft$, 
 and for $i\in \{\pm \}$ define 
 \begin{equation}\label{eq:delta-i}
  \delta_i = \deg ( \L(\Ft)|_{X_i} ),
 \end{equation}
 where $\deg$ stands for the degree with respect to the standard polarization on $X_i$. 
 Notice that for any $i$ there exists a canonical morphism $\Ft \to \L(\Ft)|_{X_i}$ from the composition 
 \begin{equation}\label{eq:canonical-quotient}
   \Ft \to \sigma^* \Ft \to \L(\Ft) \to \L(\Ft)|_{X_i}. 
 \end{equation}
 Setting 
 $$
   J(\Ft ) = \{ j\in \{0,\infty \} : \quad  \Ft \mbox{ is locally free near } j \},
 $$
 we have a short exact sequence of coherent sheaves 
 \begin{equation}\label{eq:OS-SES}
    0 \to \Ft \to \L(\Ft)|_{X_+} \oplus \L(\Ft)|_{X_-} \to \oplus_{j\in J(\Ft )} \C \to 0, 
 \end{equation}
 hence 
 \begin{equation}\label{eq:characteristic-F}
    \chi (\Ft ) + |J(\Ft )| = \chi ( \L(\Ft)|_{X_+} ) + \chi ( \L(\Ft)|_{X_-} ). 
 \end{equation}
 Applying this formula to $\Ft = \O_{X_b}$ we get 
 \begin{equation}\label{eq:characteristic-O}
   \chi(\O_{X_b} ) + 2 = \chi ( \O_{X_+} ) + \chi ( \O_{X_-} ). 
 \end{equation}
 Now subtracting (\ref{eq:characteristic-O}) from
 (\ref{eq:characteristic-F}) and taking into account definitions
 (\ref{eq:degree-F}) and (\ref{eq:delta-i}), we infer
 \begin{equation}\label{eq:delta=delta1+delta2+2-J}
   \deg (\Ft ) = \delta_+ + \delta_- + 2 - |J(\Ft )|.
 \end{equation}
 The construction of Oda and Seshadri uses the dual graph $\Gamma = (V , E)$ associated to $X_b$: 
 by definition, $V = \{ X_+,X_- \}= \{ +, - \}$  is the set of all connected components of the normalization $\tilde{X}_b$, 
 $E = \{ 0, \infty \}$ is the set of all double points of $X_b$, and an edge $j$ is adjacent to a vertex $i$ if and 
 only if the double point corresponding to $j$ lies on the connected component corresponding to $i$. 
 For $i\in \{\pm \}$ Oda and Seshadri define the value 
 $$
   d(J - J(\Ft ))_i 
 $$
 as the number of edges $j\in \{0,\infty \}$ such that $i$ is one of the end-points of $j$ and 
 $\Ft$ is not locally free at $j$. 
 As both $i= \pm$ are end-points of both edges $j\in \{0,\infty \}$, it is obvious from this definition that the quantity
 $d(J - J(\Ft ))_i$ does not depend on $i\in \{\pm \}$, and we have the equality 
 $$
   d(J - J(\Ft ))_i = |J - J(\Ft )| = 2 - | J(\Ft )|. 
 $$
 Furthermore, for any non-trivial subset $I' \subset \{\pm \}$, Oda and Seshadri set $I'' = \{\pm \} - I'$ and denote by 
 \begin{equation}\label{eq:delta-J-v}
   (\delta_{J(\Ft)}v(I''), \delta_{J(\Ft)}v(I''))
 \end{equation}
 the number of edges $j\in \{0,\infty \}$ such that $\Ft$ is locally free near $j$ and has one end-point in $I'$ and the 
 other one in $I''$. 
 As any non-trivial $I'\subset \{\pm \}$ is necessarily of the form $I' = \{ i \}$ for some $i\in \{\pm \}$ and every edge has both vertices $i$ as end-point, 
 clearly the last condition on the edges is vacuous. Hence (\ref{eq:delta-J-v}) simply gives the number of edges such that $\Ft$ is locally free near $j$, 
 said differently we find  
 $$
   (\delta_{J(\Ft)}v(I''), \delta_{J(\Ft)}v(I'')) =  |J(\Ft )|.
 $$
 With these preliminaries Oda and Seshadri call $\Ft$ $\phi$-semistable if for both $i\in \{\pm \}$ the inequalities 
 $$
   \delta_i + \frac 12 d(J - J(\Ft ))_i - \phi_i \leq \frac{(\delta_{J(\Ft)}v(I- \{ i \}), \delta_{J(\Ft)}v(I- \{ i \}))}2
 $$
 are fulfilled, and $\phi$-stable if the corresponding strict inequalities hold. 
 Plugging the formulas found above into this inequality we find that in the case of an $I_2$ curve $X_b$ 
 the semi-stability condition reads as 
 \begin{equation}\label{eq:stability-I2}
    \delta_i - \phi_i \leq |J(\Ft )| - 1, 
 \end{equation}
 and stability is defined by the corresponding strict inequality.
 Taking into account the equality of (\ref{eq:delta=delta1+delta2+2-J}),
 this may be equivalently rewritten as
 $$
   \delta - 1 < \delta_i - \phi_i \leq |J(\Ft )| - 1.
 $$
 The compactified Picard scheme 
 $$
   \cJ_{X_b}^{\delta,\phi}
 $$
 of degree $\delta \in \Z$ is then defined as the scheme parametrizing $\phi$-stable torsion-free sheaves of degree $\delta$ over $X_b$. 
 More precisely, Oda and Seshadri define the Picard functor of $\phi$-stable torsion-free sheaves and they show that it is 
 representable by a scheme.

 \subsection{Irregular Higgs bundles}\label{subsec:spectral}

 We study rank $2$ irregular Higgs bundles $(\E, \t)$ defined over $\CP1$, where $\E$ is a rank $2$ vector bundle 
 and $\t$ is a meromorphic section of $\End(\E) \otimes K$ called the Higgs field. 
 We set 
 \begin{equation*}
  \deg (\E ) = d.
 \end{equation*}
 We will limit ourselves to the case where $\t$ has a single pole $q$ of order $4$: 
 $$
  \theta : \E \to \E \otimes K(4\cdot \{ q \} ).
 $$

Introduce two local charts on $\CP{1}$: $U_1$ with $z_1\in \C$ where 
$\{z_1=0\}=q$ and $U_2$ with $z_2 \in \C$ where $\{z_2=\infty\}=q$.
Then over $\C$ the line bundle $K(4\cdot \{ q \} )$
admits the trivializing sections $\kappa_i$ over $U_i$ given as
\begin{align}
    \kappa_1 &= \frac{\d z_1}{z_1^4}, \notag\\
		\label{eq:kappa2}
		\kappa_2 &=\mathrm{d}z_2.
\end{align}
The conversion from $\kappa_1$ to $\kappa_2$ is the following:
\begin{equation}
	\label{eq:triv_conv}
	\kappa_1 = \frac{\mathrm{d}z_1}{z_1^{4}}=-z_2^2 \mathrm{d}z_2 = -z_2^2 \kappa_2.
\end{equation}
The trivialization $\kappa_i$ induces a trivialization $\kappa_i^2$ on 
$K(4\cdot \{ q \})^{\otimes 2}$, $i=1,2$.

The Hirzebruch surface $X$ can be covered by four charts. We will need
only two of those, since we only consider curves disjoint from the
section $C_{\infty}$ at infinity. Let us denote $V_i \subset p^{-1} (U_i)$ 
the complement of the section at infinity in $p^{-1} (U_i)$ ($i=1,2$). Let $\zeta \in
\Gamma \left(X ,p^* K(4\cdot \{ q \})\right)$ be the canonical section,
and introduce $w_{i} \in \Gamma (V_i,\O)$ by
\begin{equation*}
	\zeta = w_{i} \otimes \kappa_i.
\end{equation*}
Use (\ref{eq:triv_conv}) for the conversion between $w_1$ to $w_2$:
\begin{equation*}
	w_2 \otimes \kappa_2=\zeta=w_1 \otimes \kappa_1=-z_2^2 w_1 \otimes \kappa_2.
\end{equation*}

 In {the $\kappa_1$} trivialization of $\E$ near $q$ we have 
 \begin{equation}\label{eq:theta}
    \theta =  \sum_{n\geq -4} A_n z_1^n \otimes \d z_1,
 \end{equation}
 where $A_n\in \gl (2,\C )$. 

For the identity automorphism $\mbox{I}_{\E}$ of $\E$ we may consider the characteristic polynomial
 \begin{equation}\label{eq:char-poly1}
    \chi_{\theta} (\zeta ) = \det (\zeta \mbox{I}_{\E} - \theta) = \zeta^2 + \zeta F + G,
\end{equation} 
for some
 $$
  F\in H^0(\CP1 , K(4\cdot \{ q \} )), \quad G \in H^0(\CP1 , K(4\cdot \{ q \} )^{\otimes 2}). 
 $$
 Said differently, $F$ is a meromorphic differential and $G$ is a meromorphic quadratic differential. 
 
 Let us set $\vartheta_1 = \sum_{n\geq 0} A_{n-4} z_1^n$ and $\vartheta_2 = \sum_{n\geq 0} B_{n} z_2^n$, so that we have 
 $$
  \theta  = {\vartheta_i} \otimes \kappa_i,
 $$
where $i=1,2$.
 If we now factor {$\kappa_i$} in (\ref{eq:char-poly1}), then the characteristic polynomial can be rewritten as 
 \begin{equation}\label{eq:char-poly2}
    \chi_{\vartheta_i} (w_i) = \det (w_i \mbox{I}_{\E} - \vartheta_i) = w_i^2 + w_i f_i + g_i,
 \end{equation}
 with 
 $$
  F= f_i \kappa_i, \quad G = g_i \kappa_i^2.
 $$
  Now, as $K(4\cdot \{ q \} ) \cong \O(2)$, the coefficients {$f_i$ and $g_i$ are polynomials in $z_i$} of degree $2$ and $4$, respectively: 
\begin{align}
	\label{eq:f1}
	f_1(z_1) &= - (p_2 z_1^2 + p_1 z_1 + p_0), \\
	\label{eq:g1}
	g_1(z_1) &= - (q_4 z_1^4 + q_3 z_1^3 + q_2 z_1^2 + q_1 z_1 + q_0),
\end{align}
where all coefficients are elements of $\C$.
According to conversion (\ref{eq:triv_conv}):
\begin{align}
	\label{eq:f2}
	f_2(z_2) &= p_0 z_2^2+p_1 z_2+p_2,\\
	\label{eq:g2}
	g_2(z_2) &= -\left( q_0 z_2^4+q_1 z_2^3+q_2 z_2^2+q_3 z_2+q_4 \right).
\end{align}
 
 In the next two subsections we explain how to fix the polar parts of
 $\t$ depending on whether its leading-order term is regular
 semi-simple (the so-called untwisted case) or has a non-trivial
 nilpotent part (twisted case).

 \subsubsection{The untwisted case}

 In this case we will fix scalars $a_{\pm}\in \C$ with $a_+ \neq a_-$ and
 assume that the leading-order term of $\t$ (i.e., the coefficient $A_{-4}$ of
 $z_1^{-4}$ in its Laurent series) is semi-simple with eigenvalues $a_{\pm}$.
 Then there exists a polynomial gauge transformation in the indeterminate $z_1$
 that transforms $\t$ into the form
 \begin{equation}\label{eq:local-form-q1}
   \t = \left[ z_1^{-4} \begin{pmatrix}
         a_+ + b_+ z_1 + c_+ z_1^2 + \lambda_+ z_1^3 & 0 \\
         0 & a_- + b_- z_1 + c_- z_1^2 + \lambda_- z_1^3 
        \end{pmatrix} + \cdots 
  \right] \otimes \d z_1
 \end{equation}
 in some local trivialization of $\E$ near $q$ where the dots stand for higher-order matrices in $z_1$. 
 Indeed, up to applying a constant base change we may assume that $A_{-4}$ is diagonal. Furthermore the action of 
 $$
   \gamma(z_1) = \mbox{1} + \gamma_n z_1^n 
 $$
 on (\ref{eq:theta}) is 
 \begin{align*}
	\gamma(z_1) \theta(z_1) \gamma(z_1)^{-1} =& \left( A_{-4}z_1^{-4} +
 \cdots + A_{n-5} z_1^{n-5} \right. + \\
&+ \left. (A_{n-4} - \ad_{A_{-4}}(\gamma _n)) z_1^{n-4}
 + O(z_1^{n-3}) \right) \otimes \d z_1,
 \end{align*}
 and since the image of $\ad_{A_{-4}}$ is the subspace of off-diagonal
 matrices we can successively apply such gauge transformations 
 with $n=1,2$ and $3$ to cancel the off-diagonal terms of $A_{-3}$, then those of $A_{-2}$ and finally those of $A_{-1}$. 

 The matrices appearing in (\ref{eq:local-form-q1}) are called the
 \emph{polar part} of $\t$ at the singularity.  {From} now on we
 assume that the constants $a_{\pm}, b_{\pm}, c_{\pm},
 \lambda_{\pm}\in \C$ appearing in (\ref{eq:local-form-q1}) are fixed.
 A necessary condition for the existence of Higgs bundles with this
 polar part is given by the residue theorem which states that
\begin{equation}
\label{eq:residue}
   \lambda_+ + \lambda_- = 0. 
\end{equation} 
 We therefore assume that the parameters are fixed so that this equality holds. 

 We introduce 
 $$
  \P = 4 \cdot \{ q \}, \quad \P_{\red} = \{ q \}; 
 $$
 $\P$ is called the \emph{polar divisor} and $\P_{\red}$ the \emph{parabolic divisor}.
 A parabolic structure compatible with $(\E, \theta)$ is a choice 
 $$
   (\alpha_q^+, \alpha_q^-) \in [0,1)^2 
 $$ of two distinct numbers for the singular point $q \in \P_{\red}$;
  the scalars $\alpha_q^{\pm}$ are called parabolic weights.
   Essentially, $\alpha_q^{\pm}$ are associated to the $\lambda_{\pm}$
   in the above polar parts at $q$, and they correspond to the flag
 $$
   \E_q \supset L_q^+ \supset \{ 0 \}
 $$ invariant under the polar part of $\t$.  
The pair $(\alpha_q^+, \alpha_q^-)$ is \emph{generic} if 
$\alpha_q^+\alpha_q^-\neq 0$.
The parabolic weights
     constitute parameters appearing in the behavior of a compatible
     Hermitian--Einstein metric near the puncture, that one may freely
     prescribe independently of the eigenvalues of the residue of the
     Higgs field.  Notice that the associated graded $\mathfrak{t}$ of this 
     flag is a Cartan subalgebra uniquely determined by the polar part, so the only choice for the 
     parabolic structure is that of the weights $\alpha_q^{\pm}$, which then 
     singles out a Borel subalgebra containing $\mathfrak{t}$.
     A Higgs subbundle of $(\E, \t)$ is a pair $(\F, \t|_{\F})$ with $\F$ a holomorphic subbundle
   of $\E$ such that
 $$
   \t|_{\F} : \F \to \F \otimes K(\P ). 
 $$
 One immediately sees that if this is the case then the fiber $\F_q$ of $\F$ at $q$ must be one of the 
 eigenlines $L_q^{\pm}$. 
 In particular, if $(\E, \t)$ is endowed with a compatible parabolic structure then any  Higgs subbundle  
 $(\F, \t|_{\F})$ inherits a parabolic structure from $(\E, \t)$ in a natural way: according to whether 
 $\F_q = L_q^{\pm}$ we set 
 $$
   \alpha_q(\F ) = \alpha_q^{\pm}
 $$
 to be the parabolic weight of $(\F, \t|_{\F})$ at $q \in \P_{\red}$. 
 We then define 
 \begin{equation*}
    \pardeg (\E ) = \deg(\E ) + (\alpha_q^+ + \alpha_q^-) 
 \end{equation*}
 and 
 $$
   \pardeg (\F ) = \deg(\F ) + \alpha_q(\F ). 
 $$
 We say that $(\E, \t)$ is $\vec{\alpha}$-semistable if and only if for all Higgs subbundles $(\F, \t|_{\F})$ we have 
 $$
   \pardeg (\F ) \leq \frac{\pardeg (\E )}2
 $$
 and $\vec{\alpha}$-stable if strict inequality holds. Observe that if $\pardeg (\E ) = 0$ then these conditions simplify to 
 $$
   \pardeg (\F ) \leq 0
 $$
 (respectively $<$). 
 If $(\F, \t|_{\F})$ is a Higgs subbundle of $(\E, \t)$ then $\t$ also induces a morphism on the quotient vector bundle 
 $$
   \Qt = \E / \F,
 $$
 and we denote the resulting Higgs field by
 $$
   {\overline {\t}} : \Qt \to \Qt \otimes K(\P ). 
 $$
 In this situation we say that $(\Qt , {\overline {\t}})$ is a quotient Higgs bundle of $(\E, \t)$. 
 Furthermore, if $(\E, \t)$ is endowed with a compatible parabolic structure then it induces a parabolic structure on $\Qt$: 
 if $\alpha_q(\F ) = \alpha_q^{\pm}$ then we simply set 
 $$
   \alpha_q(\Qt ) = \alpha_q^{\mp}.
 $$
 Just as above, we set 
$$ 
\pardeg (\Qt ) = \deg(\Qt ) + \alpha_q(\Qt ).
$$
 By additivity of the degree, we have an equivalent definition of
 $\vec{\alpha}$-stability in terms of quotients: namely, $(\E, \t)$ is
 $\vec{\alpha}$-semistable if and only if for any quotient Higgs bundle $(\Qt,
 {\overline {\t}} )$ we have
 $$
   \pardeg (\Qt ) \geq \frac{\pardeg (\E )}2 
 $$
 and $\vec{\alpha}$-stable if strict inequality holds. Again, if $\pardeg (\E ) = 0$ then these conditions simplify to 
 \begin{equation*}
   \pardeg (\Qt ) \geq 0
 \end{equation*}
 (respectively $>$). 

 We will be interested in the moduli spaces 
 $$
   \Mod^{(s)s} = \Mod^{(s)s} (\CP1, q , a_{\pm}, b_{\pm}, c_{\pm}, \lambda_{\pm}, 
   \alpha_q^{\pm})
 $$ of $\vec{\alpha}$-stable (resp. $\vec{\alpha}$-semi-stable) irregular
   Higgs bundles on $\CP1$ of $0$ parabolic degree with the polar parts at $q$
   as prescribed in (\ref{eq:local-form-q1}), up to gauge equivalence. 
   The spaces $\Mod^{(s)s}$ are called \emph{irregular Dolbeault moduli spaces}.
   The general construction of moduli spaces $\Mod^s$ parametrizing isomorphism
   classes of stable objects was given in \cite{Biq-Boa} using gauge
   theoretic methods.  In particular, it is proved that if semi-stability is
   equivalent to stability and the adjoint orbits of the residues are closed,
   then the moduli space $\Mod^s$ is a complete hyper-K\"ahler manifold.  On
   the other hand, in order to consider moduli spaces $\Mod^{ss}$
   parametrizing equivalence classes of semi-stable objects one needs to
   slightly relax the notion of equivalence.  Namely, to any strictly
   semi-stable object $(\E,\theta )$ it is possible to find a Jordan--H\"older
   filtration
 $$
   0 \subset (\E_1,\theta_1 ) \subset (\E,\theta ) 
 $$
 (in our case necessarily of length $2$) such that both $(\E_1,\theta_1 )$ and $(\E_2,\theta_2 )$ are stable 
 (where  $\E_2 =\E / \E_1$ and $\theta_2$ is the Higgs field on $\E_2$ induced by $\theta$). 
 We then call 
 $$
   (\E_1,\theta_1 ) \oplus (\E_2,\theta_2 )
 $$ the associated graded irregular Higgs bundle of $(\E,\theta )$ and
 we call $(\E,\theta )$ and $(\E' ,\theta' )$ S-equivalent if their
 associated graded irregular Higgs bundles agree.  This definition
 reduces to isomorphism in the case of stable irregular Higgs bundles.
 We expect that there exists a quasi-projective smooth coarse moduli
 scheme $\Mod^{ss}$ parametrizing S-equivalence classes of semi-stable
 irregular Higgs bundles using a geometric invariant theory
   construction. Such a construction for the ramified irregular de
   Rham moduli space is given in \cite{Inaba}. It is highly plausible
   that the construction of Inaba carries over to provide a ramified
   irregular Dolbeault moduli space too.  In this paper we indicate an
   alternative approach to study the irregular Dolbeault moduli space.
   Namely, the relative Picard scheme was constructed by Grothendieck
   as an algebraic variety (for an exposition of the construction by
   S. Kleiman, see \cite[Theorem~9.4.8]{FGA_Pic}).  The refined
   BNR-correspondence \cite[Theorem~5.4]{Sz-bnr} is a biholomorphism
   between moduli spaces of irregular Higgs bundles of prescribed
   polar part and the Picard scheme of sheaves on ruled surfaces. The
   definition of this map is purely algebraic, hence the algebraic
   structure of the relative Picard scheme endows the complex analytic
   manifold $\Mod^{(s)s}$ with the structure of a complex algebraic
   variety. In particular, Theorems \ref{thm:mainE7} and
   \ref{thm:mainE8} provide a $\C$-analytic description of the
   corresponding moduli spaces.

 \subsubsection{The twisted case}

 We now consider the case where $A_{-4}$ has non-trivial nilpotent part. 
 In a convenient trivialization we then have 
 $$
   A_{-4} = \begin{pmatrix}
             b_{-8} & 1 \\
              0	& b_{-8}
            \end{pmatrix}
 $$
 for some $b_{-8}\in\C$ (the labeling will shortly become clear). 
 Observe that  $\im (\ad_{A_{-4}})$ is spanned by the matrices 
 $$
   \begin{pmatrix}
             0 & 1 \\
              0	& 0
            \end{pmatrix}, 
               \quad
   \begin{pmatrix}
             1 & 0 \\
              0	& -1
            \end{pmatrix}. 
 $$
 Using the same argument as in the twisted case it follows that there exists a
 polynomial gauge transformation $\gamma(z)$ 
 that transforms $\t$ into the form 
 \begin{equation}\label{eq:local-form-E8}
   \t = \left( \begin{pmatrix}
             b_{-8} & 1 \\
              0	& b_{-8}
            \end{pmatrix} z^{-4} 
            +
         \begin{pmatrix}
             0 & 0 \\
              b_{-7} & b_{-6}
            \end{pmatrix} z^{-3} 
               +
         \begin{pmatrix}
             0 & 0 \\
              b_{-5} & b_{-4}
            \end{pmatrix} z^{-2}
               +
         \begin{pmatrix}
             0 & 0 \\
              b_{-3} & b_{-2}
            \end{pmatrix} z^{-1} 
            + O(1)
            \right) \otimes \d z.
 \end{equation}
 Observe that by virtue of the residue theorem this time we have 
 $$
   b_{-2} = 0.
 $$
 On the other hand, notice that if $b_{-7} = 0$ in the above matrix then
 $A_{-4}$ can be diagonalized using the meromorphic gauge transformation 
 $$
   \gamma(z) = \mbox{1} + \begin{pmatrix} 
                       0 & -b_6^{-1} \\
                       0 & 0
                     \end{pmatrix} z^{-1}
 $$
 unless $b_{-6}$ also vanishes. 
 Since in this section we are interested in the case where $A_{-4}$ is not
 diagonalizable (even by meromorphic gauge transformations), 
 from now on we therefore assume that 
 \begin{equation*}
    b_{-7} \neq 0. 
 \end{equation*}
 and that the constants $b_{-8}, \ldots, b_{-3} \in \C$ appearing in (\ref{eq:local-form-E8}) are fixed. 

 This time the data of the parabolic structure compatible with $\theta$ is trivial, i.e. is the trivial flag 
 $$
   \E_q \supset \{ 0 \} 
 $$
 with an arbitrary weight $\alpha_q$. Indeed, as the rank of $\E$ is $2$, the only other possibility would be 
 a full flag as in the untwisted case; however, then the graded pieces of the polar parts would be of dimension $1$, 
 and we could not get nilpotent graded polar parts.

 Again, we will be interested in the moduli spaces 
 $$
   \Mod^{(s)s} = \Mod^{(s)s} (\CP1, q , b_{-8},\ldots ,b_{-3}, \alpha_q)
 $$
 of S-equivalence classes of (semi-)stable irregular Higgs bundles on $\CP1$ with polar part at $q$ with respect to some trivialization 
 as prescribed in (\ref{eq:local-form-E8}). We will see that in this case the weight $\alpha_q$ actually plays no role. 
 The existence of a moduli space parametrizing isomorphism classes of stable objects should follow from \cite{Biq-Boa}, 
 and we again expect that there should exist a quasi-projective smooth coarse moduli scheme $\Mod^{ss}$ parametrizing 
 S-equivalence classes of semi-stable objects.

 \subsection{Spectral data of irregular Higgs bundles and the irregular Hitchin map}\label{subsec:Hitchin}

 A categorical equivalence between the groupoid of irregular Higgs bundles with semi-simple polar part and the relative 
 Picard functor of a Hilbert scheme of curves on a certain multiple blow-up $\Xtt$ of the Hirzebruch surface $X$ from Subsection 
 \ref{subsec:Hirzebruch} was described in \cite{Sz-bnr}. 
 We will refer to this equivalence as the refined Beauville--Narasimhan--Ramanan (BNR-) correspondence. 
 The sheaf associated to an irregular Higgs bundle by this correspondence is called its {\emph{spectral sheaf}}, usually denoted by $\Ft$.  
 The general formula relating the degrees appearing in the two setups is 
 \begin{equation}\label{eq:delta-d}
   \delta = d + \frac 12 r (r-1) \deg ( K(4)) = d + 2, 
 \end{equation}
 where $d = \deg (\E)$ and $\delta$ denotes the degree of $\Ft$
 defined in (\ref{eq:degree-F}).  (Recall that in the latter formula
 $X_b$ denotes the support of $\Ft$.)  We refer the reader to
 \cite{Sz-bnr} for the general correspondence; in
 Subsection~\ref{subsec:RefBNR} we will spell it out explicitly in the
 untwisted case.  In the twisted case we prove an analogous result in
 Section \ref{sec:E8}.  We expect that such a result should hold in
 general, and not only in the particular case we are treating here.

 A closely related concept is {that of} the {\emph{irregular
     Hitchin map}}.  Namely, to an irregular Higgs bundle one may
 associate the support $\widetilde{\Sigma}$ of $\Ft$, called the
 \emph{spectral curve}. With the notations of Subsection
   \ref{subsec:compactified-Jacobian}, when $\widetilde{\Sigma}$ is
   singular it is an instance of one of the curves $X_b$.
 Roughly speaking, in the untwisted case it turns out that the
   prescription \eqref{eq:local-form-q1} on the eigenvalues of the
   polar parts amounts to requiring the two branches of the spectral
   curve $X_b$ to pass through the points $a_{\pm}$ in the fiber of
   $X$ over $q$ (with respect to a natural fiber coordinate), with
   first-, second- and third-order holomorphic derivatives with
   respect to $z$ equal to $b_{\pm}, c_{\pm}, \lambda_{\pm}$
   respectively.  Said differently, if one defines $\Xtt$ as the
   $8$-times blow-up of $X$ along the corresponding non-reduced
   subscheme, then the proper transform of $\widetilde{\Sigma}$
   naturally lies within $\Xtt$. Moreover, it turns out that the
   proper transform of $\widetilde{\Sigma}$ must intersect the cycles
   in second homology with prescribed intersection numbers.  To sum
   up, these conditions mean that the curve $\widetilde{\Sigma}$
   belongs to a complete linear system $|D|$ of curves on $\Xtt$
   determined by the map $\Xtt \to X$.  Finally, this curve must not
   intersect set-theoretically a given divisor (called \emph{divisor
     at infinity}); this then shows that
 the natural map
 $$
   (\E, \theta) \mapsto \widetilde{\Sigma}
 $$
 obtained by composing the refined BNR-correspondence above and the forgetful functor mapping a sheaf to its support, 
 actually takes values in an affine subspace $|D|_0 \subset |D|$. 
 For more details, see Proposition \ref{prop:E7} or \cite[Theorem 4.3]{Sz-bnr}.
 For an extension to the unramified case, see Proposition \ref{prop:equivalence}. 
 Therefore, the above association gives rise to the \emph{irregular Hitchin map}
 $$
   H : \Mod^{ss} \to |D|_0. 
 $$ 

We {call} $H$ the irregular Hitchin map {because} it is a
straightforward analogue of the map defined in \cite{Hit}.  It follows
from \cite{Biq-Boa} that for generic choices of the singularity
parameters (namely, assuming that the adjoint orbits of the residues
are closed), 
the irregular Dolbeault moduli spaces are complete
holomorphic-symplectic smooth manifolds.  Based on this fact and the
above analogy, it is therefore natural to expect that $H$ is a proper
map which endows $\Mod^{ss}$ with the structure of an algebraically
completely integrable system.

 \section{Elliptic fibrations on rational elliptic surfaces}
\label{sec:elliptic-fibrations}

 In this section we will study singular fibers of elliptic fibrations
 on rational elliptic surfaces. As 4-manifolds, these surfaces are
 diffeomorphic to the 9-fold blow-up $ \CP2 \# 9\overline{\CP{}}^2$ of
 the complex projective plane $\CP2$. The potential singular fibers
 are classified by Kodaira \cite{Kodaira}.  Here we will concentrate
 only on those fibrations which contain singular fibers of types
 ${\tilde {E}}_8$ and ${\tilde {E}}_7$. (For the plumbing description
 of these singular fibers see Figure~\ref{fig:SingFibs}.)

 \begin{figure}[hb]
 \begin{center}
\includegraphics[width=11cm]{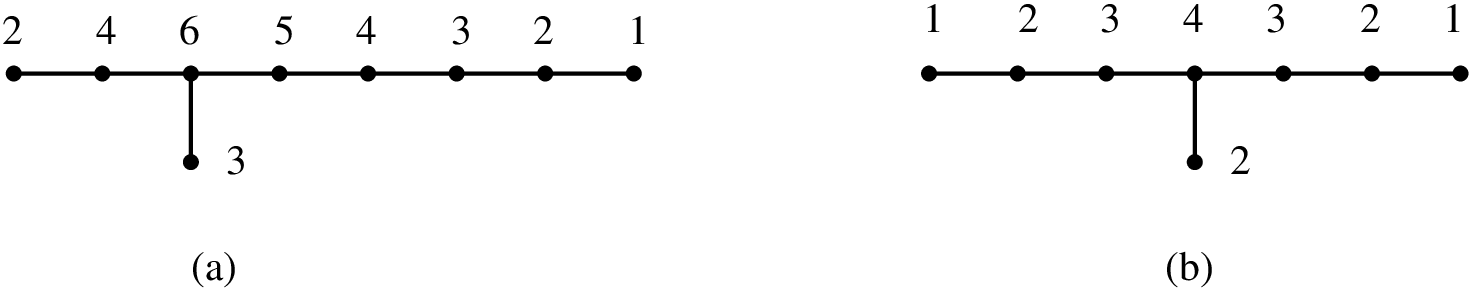}
 \end{center}
 \caption{\quad Plumbings of singular fibers of types (a) ${\tilde
     {E}}_8$ and (b) ${\tilde {E}}_7$ (integers next to vertices
   indicate the multiplicities of the corresponding homology classes
   in the fiber). All curves are rational, all intersections are
     transverse, and all self-intersections are equal to $-2$.}
\label{fig:SingFibs}
\end{figure}

One 
way to construct an elliptic fibration on the rational
 elliptic surface is by giving a pencil of cubic curves in $\CP2$ (with
 the additional property that the pencil contains at least one smooth
 cubic) and then blowing up the basepoints of the pencil. In turn, the
 pencil can be given by specifying two degree-3 homogeneous polynomials
 $p_0$ and $p_1$ in three variables and considering the curves $C(p_t)$
 corresponding to the polynomials $p_t=t_0p_0+t_1p_1$ for
 $t=[t_0:t_1]\in \CP1$.  The pencil will not contain smooth curves if
 $p_0$ and $p_1$ admit common singular points, hence  this case
will be avoided. 

 Recall that the singular fiber in an elliptic fibration with a single
 node is called $I_1$ (or a fishtail fiber), the fiber with a cusp
 singularity (which can be modeled by the cone on the trefoil knot
 $T_{2,3}$, or can be given by the local equation $y^2=x^3$) is a cusp
 fiber (also denoted by $II$). A singular fiber with two rational
 curves intersecting each other in two distinct points (and having
 self-intersection $-2$) is an $I_2$ fiber. If the two rational
 curves are tangent to each other (still with self-intersection $-2$)
 then we have a type $III$ fiber.  (There are further singular fibers
 in the Kodaira list, but we will not meet them in our subsequent arguments.)

 The determination of the type of all singular fibers in an elliptic
 fibration specified by two cubic polynomials $p_0, p_1$ can be 
 a rather tedious problem. By choosing specific polynomials, the existence of
 two singular fibers is quite transparent, but the identification of the further
 ones usually requires further computations.

\subsection{The case of singular fibers of type \texorpdfstring{$\Et8$}{E\texteightinferior}}
\label{ss:e8}
 Suppose first that we have an elliptic fibration on $\CP2 \# 9
 {\overline {\CP{}}}^2$ with a singular fiber of type $\Et8$. We will
 also assume that the fibration comes from blowing up a pencil, hence
 it admits a section. This section then necessarily intersects the
 $\Et8$-fiber in the unique curve with multiplicity 1. Consider a
 generic fiber $C$ of the fibration, and blow down the section and then
 consecutively the next six curves of the $\Et8$-fiber. The image of
 $C$ (now of self-intersection 7) will intersect two curves $E_1, E_2$
 (both of self-intersection $(-1)$) from the fiber, one of which (say
 $E_2$) is further intersected by the leaf $E_3$ of the $\Et8$ fiber,
 and is of multiplicity 2.  (We point out that, as it is obvious from
 the construction, the two curves $E_1, E_2$ intersect $C$ at the same
 point, cf. the left diagram of Figure~\ref{fig:conf}.)

 \begin{figure}[hb]
 \begin{center}
 \includegraphics{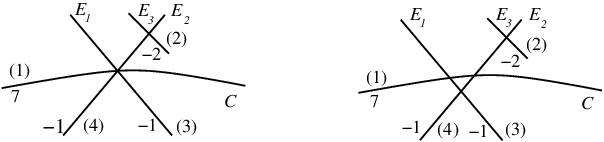}
 \end{center}
 \caption{\quad Curve configurations when blowing down a section and 
 a singular fiber of type (a) $\Et8$ and (b) $\Et7$.
 Integers next to the curves indicate self-intersections, while integers in 
 brackets are multiplicities.}
 \label{fig:conf}
 \end{figure}

 There is a choice in continuing the blow-down process. If we blow
 down $E_1$, then we get a configuration of curves in the second
 Hirzebruch surface, where the image of $E_2$ is a fiber, $E_3$ is the
 section at infinity, and $C$ blows down to a multisection,
 intersecting the generic fiber twice and being tangent to $E_2$. On
 the other hand, blowing down $E_2$ first, and then $E_3$, the curve
 $C$ blows down to a cubic curve $C_0$ in $\CP2$, and the image of
 $E_1$ will be a projective line, triply tangent to $C_0$ (at one of
 its inflection points).  The two results are related by the
 birational morphism $\omega$ of Equation~\eqref{eq:birational}.

 In conclusion, 
\begin{thm} \label{thm:e8fibration}
Any elliptic fibration on $\CP2\# 9 {\overline
   {\CP{}}}^2$ with a section and with a singular fiber of type $\Et8$
 can be blown up from a pencil defined by either 
\begin{enumerate}
\item the union of the infinity section
 (with multiplicity 2) with a fiber (with multiplicity 4) in
 the second Hirzebruch surface, and with a double section which is
 tangent to the chosen fiber, or 

\item a cubic curve in $\CP2$, with a triple
 tangent line (at one of the inflection points of the cubic), the latter with
 multiplicity three. 
\end{enumerate}
The converse statements also hold: pencils given by (1) or
(2) above give rise to fibrations (after the infinitely close blow-ups
of the base point) to elliptic fibrations containing an $\Et8$ fiber.\qed
\end{thm}

\subsection{The case of singular fibers of type \texorpdfstring{$\Et7$}{E\textseveninferior}}
\label{sec:singfib_E7}
 Next we would like to analyze pencils resulting in fibrations with singular
 fibers of type $\Et7$. Assume therefore that the fibration on $\CP2\# 9
 {\overline {\CP{}}}^2$ contains such a singular fiber, and that the fibration
 results from a pencil, hence it also admits a section. Indeed, since the
 pencil should have at least two basepoints (otherwise the fibration has a
 singular fiber which contains a chain of 8 curves with self-intersection
 $(-2)$, which is impossible next to a fiber of type $\Et7$), we can assume
 that there are two sections, intersecting the type $\Et7$ singular fibers in
 the two $(-2)$-curves with multiplicity 1. As before, let $C$ be a regular
 fiber of the fibration.

 After 7 blow-downs (by blowing down the two sections and two,
 respectively three curves from the two long arms of the $\Et7$-fiber)
 we get a configuration of 4 curves: the image of the fiber $C$, two
 $(-1)$-curves (called $E_1$ and $E_2$) intersecting it in two
 distinct points (and also intersecting each other) and a $(-2)$-curve
 $E_3$ intersecting $E_2$ only, cf. the right diagram of
 Figure~\ref{fig:conf}. As in the case of an $\Et8$-fiber, we have a
 choice in performing the next blow-down. If we blow down $E_1$, we
 get a configuration again in the second Hirzebruch surface, while if
 we blow down $E_2$ (and then $E_3$), we get a configuration in
 $\CP2$.  Consequently we get
\begin{thm}\label{thm:e7fibration}
Any elliptic fibration on $\CP2\# 9 {\overline
   {\CP{}}}^2$ with two sections and with a singular fiber of type $\Et7$
 can be blown up from a pencil defined by either 
\begin{enumerate}
\item the union of the infinity section
 (with multiplicity 2) with a fiber (with multiplicity 4) in
 the second Hirzebruch surface, and with a double section which intersects
the distinguished fiber in two distinct points, or
\item
a cubic in $\CP2$, with a tangent line which intersects the cubic in
one further point; the tangent line with multiplicity three. 
\end{enumerate}
The converse of this statement also holds:
the pencils specified in (1) or (2) above --- after infinitely close
blow-ups of the base points --- give rise to elliptic fibrations 
containing an $\Et7$ fiber.\qed
\end{thm}

Assume now that the elliptic fibration contains (besides the type
$\Et7$-fiber) a further singular fiber which is either of type $I_2$
or of type $III$.  By further inspecting the blow-down process, now
choosing the curve $C$ to be a singular fiber of type $I_2$ or $III$
we get:
\begin{prop}\label{prop:DoublesectionAnd2section}
If an elliptic fibration with a fiber of type $\Et7$ and two sections 
contains a further singular fiber either of type $I_2$ or of type $III$,
then the pencil of curves resulting from the repeated blow-down in the
second Hirzebruch surface contains a double section which is the union of
two sections of the ruling of the surface. \qed
\end{prop}

The same argument (now by blowing down the configuration to $\CP2$)
shows that the pencil in $\CP2$ can be chosen to be generated by a
projective line $\ell$ (with multiplicity three, just as before) and
another curve, which has two components, a line $\ell _1$ and a
quadric $q$, where $\ell$  intersects $\ell _1$ in one point $P$, while
$\ell$ is tangent to the quadric $q$ (in a point distinct from $P$). The
pencil gives rise to a fibration which has (besides a type $\Et7$ fiber)
an $I_2$ fiber if $\ell _1$ intersects $q$ in two distinct points, and 
a type $III$ fiber if $\ell _1$ is tangent to $q$.

\section{The untwisted case}\label{sec:E7}

\subsection{The refined BNR-correspondence}
\label{subsec:RefBNR}
 We start by applying the refined BNR-corre\-spon\-dence of
   \cite{Sz-bnr} to describe a certain blow-up $\Xtt$ of the surface
 $\Xt$ whose geometry governs $\Mod$.  {We have already referred to
   $\Xtt$ in Subsection \ref{subsec:Hitchin}; here we will make its
   construction rigorous.}  Namely, a local trivialization of $K(4)
 \cong K\otimes \O( 4 \cdot \{ q \})$ near $z_1 = 0$ is given by
 $z_1^{-4}\d z_1$, so the expressions $z_1^{-4}(a_{\pm} + b_{\pm} z_1
 + c_{\pm} z_1^2 + \lambda_{\pm} z_1^3) \d z_1$ specify non-reduced
 subschemes of dimension $0$ and length $4$ in {$X$}. We define
 $\Xtt$ as the blow-up of {$X$} along these subschemes, with
   $\Xt$ being an intermediate step in the blowing up.
	
 In concrete terms, as in Section \ref{sec:prep-mat} $q$ denotes the
 point with $z_1=0$, ${U_1} = \C = \CP1 \setminus \{ \infty \}$,
 ${\kappa_1} = z_1^{-4}\d z_1$, and parametrize $p^{-1}
 ({U_1}) \setminus C_{\infty}$ by coordinates $(z_1, w_1)\in
 \C^2$ as follows: we let the point of $X$ corresponding to these
 parameters be $ [w_1 \kappa_{1} : \mathbf{1}]$. We may assume that
 $\Xt$ is the blow-up of $X$ in the point $ [a_+ \kappa_1(0) :
   \mathbf{1} ]$, i.e.  over $ p^{-1} ({U_1})$ the surface $\Xt$
 is defined by
 $$
   ( z_1 w_1' - (w_1 - a_+) z_1' ) \subset \C^2 \times \CP1
 $$
 where $[z_1' : w_1'] \in \CP1$ are homogeneous coordinates corresponding to the direction of tangent vectors at $z_1 = 0, w_1 = a_+$. 
 We denote this blow-up by 
 $$
   \sigma_{1+} : X_{1+} = \Xt \to X
 $$
 and its exceptional divisor by 
 $$
     E_{1+} = \{ z_1 = 0, w_1 = a_+, [z_1' : w_1'] \} .
 $$
 According to \cite[(4.25)]{Sz-bnr}, we now need to blow up $\Xt$ in the point 
 $$
   [z_1' : w_1' ] = [1 : b_+ ] \in E_{1+}. 
 $$
 For this purpose, we introduce the local chart ${U_{1+}'}$ of $\Xt$ given by $z_1'\neq 0$. 
 Here we may normalize $z_1' = 1$, and so a local coordinate chart of ${U_{1+}'}$ is given by $z_1,w_1'$. 
 The blow-up 
 $$
   \sigma_{2+} : X_{2+} \to X_{1+}
 $$
 we consider is then the blow-up of the point with coordinates $z_1 = 0, w_1' = b_+$. 
 Similarly to the above, we denote the exceptional divisor of $\sigma_{2+}$ by $E_{2+}$, 
 and we get canonical coordinates $[z_1'':w_1'']$ parametrizing $E_{2+}$ starting from 
 the coordinates $z_1, z_1'$. 
 Again by \cite[(4.25)]{Sz-bnr}, we now blow up the point 
 $$
   [z_1'' : w_1'' ] = [1: c_+ ] \in E_{2+}
 $$
 and call the corresponding birational map 
 $$
  \sigma_{3+} : X_{3+} \to X_{2+}. 
 $$
 Finally, just as above we get canonical coordinates $[z_1''':w_1''']$ on the exceptional divisor 
 $E_{3+}$ of $\sigma_{3+}$, and we define the blow-up 
 $$
   \sigma_{4+} : X_{4+} \to X_{3+} 
 $$
 of the point with coordinates 
 $$
   [z_1''' : w_1''' ] = [1 : \lambda_+ ] \in E_{3+}. 
 $$
 We then let $X_{0-} = X_{4+}$ and carry out a similar procedure for the length $4$ non-reduced subschemes 
 corresponding to the expression $z_1^{-4}(a_{-} + b_{-} z_1 + c_{-} z_1^2 + \lambda_{-} z_1^3) \d z_1$. 
 We denote the birational maps and their exceptional divisors by 
 $$
   \sigma_{i-} : X_{i-} \to X_{(i-1)-}
 $$
 and $E_{i-}$ for $1\leq i \leq 4$. 
 By an abuse of notation, we will continue to denote the proper transforms of $E_{i+}$ and $E_{i-}$ along the subsequent maps 
 $\sigma_{j+}$ and $\sigma_{j-}$ by the same symbols. 
 The surface of interest to us is 
 \begin{equation}\label{eq:Xtt-untwisted}
    \Xtt = X_{4-} \xrightarrow{\sigma} X. 
 \end{equation}
 Clearly then there is a diagram 
 $$
   \xymatrix{
     & \Xtt \ar[ld] \ar[rd] & \\
   X \ar@{..>}[rr]^{\omega} && \CP2}
 $$ where the left-hand map is a blow-up of $X$ in $8$ points and the
   right-hand map is a blow-up of $\CP2$ in $9$ points.  In particular,
   as a smooth $4$-manifold $\Xtt$ is diffeomorphic to $ \CP2 \#
   9\overline{\CP{}}^2 $. By an abuse of notation, we will denote the
   composition of $\Xt \to X$ with $p : X \to \CP1$ by $p : \Xt \to
   \CP1$ and also the composition of $\Xtt \to X$ with $p : X \to
   \CP1$ by $p : \Xtt \to \CP1$.

 It follows from \cite[Theorem~4.3]{Sz-bnr} that irregular rank $2$
 Higgs bundles on $\CP1$ with a pole of order $4$ of the local form
 (\ref{eq:local-form-q1}) are in one-to-one correspondence with data
 of the form $(\St, \Ft)$ where $\St$ is a closed holomorphic curve in
 $\Xtt$ satisfying certain properties and $\Ft$ is a torsion-free
 sheaf of $\O_{\St}$-modules of some given degree $\delta$.  
\begin{defn}\label{def:Hitchin_base}
 Let $|D|_0$ denote the set of closed holomorphic curves in $\Xtt$ satisfying 
the following three conditions:
\begin{enumerate}[label=(\alph*)]  
  \item $\St$ is disjoint from the proper transform of $C_{\infty}$ in
    $\Xtt$; \label{item:SpectralCurveCond1}
  \item $p:\St \to \CP1$ is a double ramified
    cover; \label{item:SpectralCurveCond2}
  \item $\St$ intersects the exceptional divisors $E_{4\pm}$ in one
    point each, away from their ``points at infinity'' $[z_1^{(iv)} :
      w_1^{(iv)}] = [ 0 : 1 ]\in
    E_{4\pm}$. \label{item:SpectralCurveCond3}
\end{enumerate}
\end{defn}
In particular, conditions \ref{item:SpectralCurveCond2}--\ref{item:SpectralCurveCond3} imply that any $\St\in |D|_0$ intersects neither the proper 
 transform $\tilde{F}_0$ of the fiber $F_0$ in $\Xtt$ nor the exceptional divisors $E_{i\pm}$ with $1\leq i \leq 3$.

\begin{prop} \label{prop:E7}
 There exists an elliptic fibration 
 $ \Xtt \to \CP1$
 with an $\Et7$ singular fiber $\Xtt_{\infty}$ over $\infty \in \CP1$, such that 
 $\Mod ^{ss}$ is a relative compactified Picard scheme of
 torsion-free sheaves of relative degree $1$ over $\Xtt \setminus \Xtt_{\infty}$. 
\end{prop}

\begin{proof}
 Let $F$ denote the fiber class of the Hirzebruch surface, 
 $\tilde{F}_0$ the proper transform under the map~\eqref{eq:Xtt-untwisted} of the fiber $F_0$ of $p$ over $q$, and
 recall again our convention that $E_{i \pm}$ stands for the proper
 transform in $\Xtt$ of the exceptional divisor of the blow-up
 $\sigma_{i \pm}$.  
 The Picard group of $\Xtt$ is generated by the classes $F, C_{\infty}, E_{i\pm} (1\leq i \leq 4)$, 
 with only non-zero intersection numbers among these classes  
 \begin{align*}
  C_{\infty}^2 & = -2 \\
  F \cdot C_{\infty} & = 1 \\
  E_{i\pm}^2  & = - 2 \quad (1\leq i \leq 3) \\
  E_{4\pm}^2  & = - 1 \\
  E_{i +}\cdot E_{(i+1) +} & = 1 \quad (1\leq i \leq 3)\\
  E_{i -}\cdot E_{(i+1) -} & = 1 \quad (1\leq i \leq 3). 
 \end{align*}
 We note the relation 
 \begin{equation}\label{eq:F0}
    F = \tilde{F}_0 + \sum_{i=1}^4 (E_{i +} + E_{i -}). 
 \end{equation}
 Consider the divisor 
 \begin{equation*}
  \Xtt_{\infty} = 2 C_{\infty} + 4\tilde{F}_0 + 3 (E_{1+} + E_{1-}) +
  2 (E_{2+} + E_{2-}) + (E_{3+} + E_{3-})
 \end{equation*}
 of $\Xtt$ and the linear system $|D|$ generated by $\Xtt_{\infty}$ in
 $\Xtt$. A straightforward check using the above intersection numbers
 shows that $\Xtt_{\infty}$ is of type $\Et7$, in particular its
 self-intersection number is $0$.

For completing the proof of the proposition, we need a few lemmas.

 \begin{lem}\label{lem:intersection}
  A projective curve $\St \subset \Xtt$ belongs to $|D|$ if and only if 
 \begin{itemize}
  \item $\St \cdot C_{\infty} = 0$; 
  \item $\St \cdot F = 2$; 
  \item $\St\cdot E_{4+} = 1 = \St\cdot E_{4-}$. 
 \end{itemize}
 \end{lem}
 \begin{proof}
 An easy check shows that for $\St = \Xtt_{\infty}$, the algebraic intersection numbers satisfy all the asserted requirements.
 For any curve $\St \in |D|$ the line bundles $\O_{\Xtt}(\St)$ and $\O_{\Xtt}(D)$ are linearly equivalent. 
 On the other hand, for any other projective curve $C\subset \Xtt$ we have 
 $$
  \St\cdot C = \langle c_1(\O_{\Xtt}(\St)), [C] \rangle.
 $$
 Since the first Chern class only depends on the linear equivalence class, the above observation implies the ``only if'' direction. 
 
 For the other direction, note that any curve $\St$ with given intersection numbers is homologous to $\Xtt_{\infty}$ because the intersection lattice of 
 $\Xtt$ is non-degenerate and generated by $F, C_{\infty}, E_{i\pm} (1\leq i \leq 4)$. 
 Said differently, the line bundles $\O_{\Xtt}(D)$ and $\O_{\Xtt}(\St)$ have the same first Chern class 
 \begin{equation}\label{eq:c1}
  c_1 (\O_{\Xtt}(D)) = c_1 (\O_{\Xtt}(\St )).  
 \end{equation}
 Now, the Picard group $\Pic(\Xtt )$ can be written as an extension 
 $$
  0 \to \Pic^0(\Xtt ) \to \Pic(\Xtt ) \xrightarrow{c_1} H^2 (\Xtt, \Z ) \to 0
 $$
 with 
 $$
  \Pic^0(\Xtt ) = H^{1,0}(\Xtt) / H^1(\Xtt, \Z). 
 $$
 Taking into account that $H^1(\Xtt, \C) = 0$, this implies that $\Pic^0(\Xtt ) = 0$. Then \eqref{eq:c1} implies that $\O_{\Xtt}(D) = \O_{\Xtt}(\St)$.  
 \end{proof}

 The conditions of Lemma~\ref{lem:intersection} are counterparts in
 terms of algebraic intersection numbers of the geometric
 conditions~\ref{item:SpectralCurveCond1}--\ref{item:SpectralCurveCond3}
 of Definition~\ref{def:Hitchin_base}.  (Just as there, it follows
 from these requirements and the relation~\eqref{eq:F0} that $\St\cdot
 E_{i\pm} = 0$ for all $1\leq i \leq 3$.)  From this, we see that
 $|D|_0\subseteq |D|$. The base of $|D|$ is $P(H^0(\Xtt,
 \O_{\Xtt}(D)))$.

 \begin{lem}\label{lem:pencil}
  We have $\dim_{\C} H^0(\Xtt, \O_{\Xtt}(D)) = 2$, i.e. $|D|$ is a pencil. 
 \end{lem}
 \begin{proof}
  Consider the short exact sequence 
  $$
    0 \to \O_{\Xtt} \to \O_{\Xtt}(D) \to \O_{D}(D) \to 0
  $$
  of sheaves on $\Xtt$, and its associated long exact sequence in cohomology 
  $$
    0 \to H^0(\Xtt, \O_{\Xtt}) \to H^0(\Xtt, \O_{\Xtt}(D)) \to  H^0(\Xtt, \O_{D}(D)) \to H^1(\Xtt, \O_{\Xtt}) = 0.
  $$
  Since $D\cdot D =0$, we have 
  $$
     H^0(\Xtt, \O_{D}(D)) =  H^0(\Xtt, \O_{D}) = H^0(D, \O_{D}) \cong \C.
  $$
  This implies the assertion. 
 \end{proof}

 \begin{lem}\label{lem:ag}
  Let $\St\in |D|_0$.  Then, 
\begin{enumerate}
\item \label{item:CP2} the restriction of the birational map
  $\Xtt \to \CP2$ establishes a biholomorphism between $\St$ and a
  cubic curve in $\CP2$;  
\item \label{item:X} the restriction of the birational map \eqref{eq:Xtt-untwisted} establishes a
  biholomorphism between $\St$ and a closed holomorphic curve in $X$.
\end{enumerate}
In particular, by \eqref{item:CP2} $\St$ is of arithmetic genus $1$.
\end{lem}
 \begin{proof}
  Under the map $\Xtt \to \CP2$ the generic fibers of $p:\Xtt \to
  \CP1$ get mapped to curves of self-intersection number $1$, i.e. to
  lines $\ell$ in $\CP2$ passing through $[0:0:1]$.  Thus the image of
  a curve $\St$ is a curve in $\CP2$ intersecting the generic such
  line $\ell$ in two points distinct from $[0:0:1]$ (corresponding to
  the intersection points of $\St$ with the generic fiber of $\Xtt$).
  Furthermore it is easy to see that the point $[0:0:1]$ is a base
  point of such curves $\St$, but blowing it up once is sufficient to
  separate them.  In different terms, $\St$ intersects the generic
  line $\ell$ passing through $[0:0:1]$ in $3$ points (counted with
  multiplicity).  By the conditions, no component of $\St$ gets
  contracted to a point and moreover no two points of $\St$ get
  identified. We infer that the restriction is one to one. 
  This proves part \eqref{item:CP2}.
  
  For part \eqref{item:X}, it is sufficient to prove that the centers of the quadratic transformations 
  $\sigma_{i \pm}$ are smooth points of $\sigma(\St)$ and its proper transforms. This immediately follows as 
  $\sigma(\St)$ transversely intersects the fiber of $X$ over $q$ in two distinct points. 
 \end{proof}

 \begin{lem}
  The map $\Xtt \to |D|$ is a fibration. 
 \end{lem}
 \begin{proof}
  The union of the curves $\St \in |D|$ is of dimension $2$, so it is
  equal to $\Xtt$ because this latter is irreducible.  
Since the curves in $|D|$ have zero self-intersection, the pencil is 
indeed a fibration.
 \end{proof}

 \begin{lem}
  The curve $\Xtt_{\infty}$ is the only element of $|D|\setminus |D|_0$. 
 \end{lem}
 \begin{proof}
 The curve $E_{4+}$ results from the last blow-up, it is a section of
 the elliptic fibration $\Xtt \to |D|$. Through any point of $E_{4+}$
 there passes a unique curve $\St \in |D|$.  Now, $\Xtt_{\infty}$ is
 the curve passing through the point $[ 0 : 1 ]\in E_{4+}$.
 Therefore, any fiber $\St \in |D| \setminus \{ \Xtt_{\infty} \}$
 intersects $E_{4+}$ transversely in a point different from $[ 0 : 1
 ]\in E_{4+}$, and is distinct from the fiber $\Xtt_{\infty}$.  This
 shows that the geometric
 conditions~\ref{item:SpectralCurveCond1}--\ref{item:SpectralCurveCond3}
 listed in Definition~\ref{def:Hitchin_base} are fulfilled.
 \end{proof}

With the above lemmas at hand, now we are ready to return to the
  proof of Proposition~\ref{prop:E7}.  \cite[Theorem~4.3]{Sz-bnr} now
implies that $\Mod ^{ss}$ is a relative compactified Picard scheme of
torsion-free sheaves of relative degree $1$ over
$  \Xtt \setminus \Xtt_{\infty} \to |D|_0.$ This concludes the proof of
Proposition~\ref{prop:E7}.
\end{proof}

\subsection{Local description of irregular Higgs bundles}
\label{subsec:E7_local}
Next we will start identifying the singular fibers of the resulting
elliptic fibration. 
In the untwisted case, the matrices in (\ref{eq:local-form-q1}) 
gave a local form for $\t$. 
The matrix $A_{-4}$ will encode the base locus of a pencil 
and the matrices $A_{-3}, A_{-2}$ and $A_{-1}$ will represent the 
tangents and the higher order derivatives of the curves of a pencil.

We wrote $\chi_{\vartheta_i} (w_i)$ as the characteristic
  polynomial of $\t$ in the trivialization given by $\kappa_i$ and
  $\kappa_i^2$ ($i=1,2$), see~\eqref{eq:char-poly2}, cf. also
  Subsection~\ref{subsec:spectral}.  The polynomials
  $\chi_{\vartheta_i} (w_i)$ are the local forms of the spectral
  curves in $X$.  In concrete terms, using
  Equations~\eqref{eq:f1},~\eqref{eq:g1},~\eqref{eq:f2} and
  \eqref{eq:g2}:
\begin{align}
	\chi_{\vartheta_1} (z_1, w_1) & = w_1^2 - \left(p_2 z_1^2+p_1 z_1+p_0\right) w_1 -\left(q_4 z_1^4+q_3 z_1^3+q_2 z_1^2+q_1 z_1+q_0\right), \notag \\
	\label{eq:char-poly3_2}
	\chi_{\vartheta_2}  (z_2, w_2) & = w_2^2 + \left(p_0 z_2^2+p_1 z_2+p_2\right) w_2 - (q_0 z_2^4+q_1 z_2^3+q_2 z_2^2+q_3 z_2+q_4).
\end{align}

The roots of the characteristic polynomial in $w_1$ have expansions
with respect to $z_1$ {near $q$}. The first several terms of the
expansion are the same as the diagonal elements of the matrix in
(\ref{eq:local-form-q1}). More precisely, the series of the
''negative'' root of $\chi_{\vartheta} (w_1)$ up to third order is
equal to $a_- + b_- z_1 + c_- z_1^2 + \lambda_- z_1^3$ and the
''positive'' root up to third order is equal to $a_+ + b_+ z_1 + c_+
z_1^2 + \lambda_+ z_1^3$. From these equations we get the following
expressions:
\begin{align*}
	f_1(z_1) =& - \left( \left(c_- + c_+\right) z_1^2 + \left(b_- + b_+\right) z_1 + (a_- +a_+)\right), \\
	\begin{split}
	g_1(z_1) =& - q_4 z_1^4+ \left(a_- \lambda_+ + a_+ \lambda_- +
        b_+ c_- + b_- c_+\right) z_1^3 + \\ &+ \left( a_+ c_- + a_-
        c_+ + b_- b_+ \right) z_1^2 + \left(a_+ b_-+a_- b_+\right) z_1
        + a_- a_+.
	\end{split}
\end{align*}
According to the residue theorem (\ref{eq:residue}) we know that
$\lambda _+ + \lambda _-=0$, hence we can eliminate 
$\lambda_-=-\lambda_+$. 
It turns out that these equations do not depend on $q_4$ (the
  coefficient of $g_1$ and $g_2$), thus we set
$$
	t=q_4.
$$
Hence we get a pencil parametrized by $t$ with base locus 
$(0,a_+)$ and $(0, a_-)$ in $\C^2$: 
\begin{align*}
\begin{split}
 \chi_{\vartheta_1} (z_1, w_1, t) =& w_1^2 - \left(\left(c_-+c_+\right) z_1^2 +\left(b_-+b_+\right) z_1 + a_-+a_+ \right) w_1 -\\
	& - t z_1^4 + \left(\left(a_--a_+\right) \lambda _++b_+ c_-+b_- c_+\right) z_1^3 + \\
	& + \left(a_+ c_-+a_- c_++b_- b_+\right) z_1^2 + \left(a_+ b_- + a_- b_+\right) z_1 + a_- a_+ = 0.
\end{split}
\end{align*}
We note that $\chi_{\vartheta_1} (z_1, w_1, t)$ intersects the fiber component of 
the fiber with multiplicity $4$ in two distinct points 
and every spectral curve is a double section.

If we rewrite the Equation~\eqref{eq:char-poly3_2}, then we get a pencil on the chart $U_2$:
\begin{align*}
\begin{split}
 \chi_{\vartheta_2} (z_2, w_2, t) =& w_2^2 + f_2(z_2) w_2 + g_2(z_2,t) = \\
	=& w_2^2 + \left(\left(a_-+a_+\right) z_2^2+\left(b_-+b_+\right) z_2+c_-+c_+\right) w_2 +\\
	& + a_- a_+ z_2^4 + \left(a_+ b_-+a_- b_+\right) z_2^3 + \left(a_+ c_-+a_- c_++b_- b_+\right) z_2^2 +\\
	& + \left(\left(a_--a_+\right) \lambda _++b_+ c_-+b_- c_+\right) z_2 -t = 0.
\end{split}
\end{align*}

More precisely, the pencil in the Hirzebruch surface $X$ is defined by 
$\chi_{\vartheta_1} (z_1, w_1, t)$ 
and the union of the section at infinity with fiber $F_0$. 
According to the converse direction of Theorem~\ref{thm:e7fibration}, the
pencil gives rise to an elliptic fibration in
$\CP2\# 9 {\overline {\CP{}}}^2$ with a singular fiber of type
$\Et7$. 

Our goal is to find the other singular fibers in the
  pencil. For this reason, we will identify the singular points on the
  spectral curves.  The spectral curves intersect the fiber component
  of the curve $C_{\infty}$ at infinity (whose fiber component is of
  multiplicity $4$) in two distinct points and according to Condition
  \ref{item:SpectralCurveCond3} of Definition~\ref{def:Hitchin_base},
  the pencil has no singular point on the distinguished fiber $F_0$.
  Thus it is sufficient to consider the $\kappa_2$ trivialization,
  i. e. the chart $(z_2, w_2)$. For identifying the singular fibers in
  the pencil, we look for triples $(z_2, w_2,t)$ such that $(z_2,
  w_2)$ fits the curve with parameter $t$ and the partial derivatives
  below vanish:
\begin{subequations}
\label{eq:partials}
\begin{align}
	\label{eq:partials_first}
	\chi_{\vartheta_2} (z_2, w_2, t) &= 0, \\
	\label{eq:partials_second}
	\frac{\partial \chi_{\vartheta_2}(z_2, w_2, t)}{\partial w_2} &=0, \\
	\label{eq:partials_third}
	\frac{\partial \chi_{\vartheta_2}(z_2, w_2, t)}{\partial z_2} &=0.
\end{align}
\end{subequations}
These triples are in one-to-one correspondence with singular points of
singular fibers. Every spectral curve $X_b$ is a double section of the
ruling on the Hirzebruch surface $X$, thus every triple $(z_2, w_2,t)$
satisfying Equations~\eqref{eq:partials} maps to distinct points under
the ruling~$p$. Indeed, if one fiber (with fixed $t$ value) contains
two singular points with the same $z_2$ coordinate then the
corresponding fiber of $p$ would intersect $X_b$ with multiplicity
higher than two. Furthermore, it cannot happen that two singular
points with the same $z_2$ coordinate lie on distinct fibers (two
distinct $t$ values): we will see in Equation~\eqref{eq:E7_t} that the
$t$ values are determined by the $z_2$ values. Consequently the
$z_2$-values from the triples $(z_2, w_2,t)$ are in one-to-one
correspondence with singular points.

Computing the partial derivatives and expressing $w_2$ from 
Equation~\eqref{eq:partials_first} and $t$ from 
Equation~\eqref{eq:partials_second} by $z_2$ we get:
\begin{align}
	w_2 (z_2)=& -\frac{1}{2} \left(\left(a_-+a_+\right) z_2^2+\left(b_-+b_+\right) z_2+c_-+c_+\right), \notag \\
	\begin{split}\label{eq:E7_t}
	t (z_2) =& -\frac{1}{4} \left(\left(a_--a_+\right){}^2 z_2^4 + 2 \left(a_--a_+\right) \left(b_--b_+\right) z_2^3\right. +\\
	&+ \left(2 \left(a_--a_+\right) \left(c_--c_+\right)+\left(b_--b_+\right){}^2 \right) z_2^2 +\\
	&+ \left.\left( 2 \left(b_--b_+\right) \left(c_--c_+\right)- 4 \left(a_- - a_+\right) \lambda _+ \right) z_2 + \left(c_-+c_+\right){}^2 \right).
	\end{split} 
\end{align}
Substitute the resulting expression into the Equation~\eqref{eq:partials_third} and get
\begin{equation}
\begin{split}
	0=& 2 \left(a_--a_+\right){}^2 z_2^3+3 \left(a_--a_+\right) \left(b_--b_+\right) z_2^2 + \\
	&+ \left(2 \left(a_--a_+\right) \left(c_--c_+\right)+\left(b_--b_+\right){}^2\right)z_2 + 2 \left(a_+-a_-\right) \lambda _+ +\left(b_--b_+\right) \left(c_--c_+\right).
	\end{split}
	\label{eq:E7_cubic} 
\end{equation}

The roots of this polynomial correspond to the $z_2$ values of
  the singular points in the singular curves on the Hirzebruch surface
  $X$, which become fibers on the 8-fold blow up.  Since we have a
  cubic polynomial in Equation~\eqref{eq:E7_cubic}, generally we get
  three distinct roots, and this corresponds to the fact that there
  are at most three singular fibers in the fibration (next to
  $\Et7$).

The cubic polynomial of \eqref{eq:E7_cubic} with variable 
$z_2$ has multiple roots if and only if its discriminant 
\begin{equation*}
	\left(a_--a_+\right){}^2 \left(\left(\left(b_--b_+\right){}^2-4 \left(a_--a_+\right) \left(c_--c_+\right)\right){}^3- 432 \left(a_--a_+\right){}^4 \lambda _+^2\right)
\end{equation*}
vanishes.

With the choice $a_- = a_+$ the configuration  reduces to the case of a
type  $\Et8$ singular fiber (to be treated in Section~\ref{sec:E8}), therefore we can assume that
$a_- \neq a_+$. 
We define 
\begin{equation*}
	\Delta:=\left(\left(b_--b_+\right){}^2-4 \left(a_--a_+\right) \left(c_--c_+\right)\right){}^3- 432 \left(a_--a_+\right){}^4 \lambda _+^2.
\end{equation*}

We analyze the cases depending on how many singular points are in the
fibration.

\subsubsection{One root}

The cubic in (\ref{eq:E7_cubic}) has one root  if and only if the
discriminant $\Delta$ vanishes, and the derivative of~
(\ref{eq:E7_cubic}) with respect to {$z_2$} has one root, 
hence the discriminant of the latter quadratic equation also vanishes.
This means that {$\Delta_0=0$} with 
\begin{equation*}
	\Delta_0 := 3 \left(a_--a_+\right){}^2 \left( \left(b_--b_+\right){}^2 - 4 \left(a_--a_+\right) \left(c_--c_+\right)\right).
\end{equation*}
It is easy to see $\Delta={\Delta_0}=0$ is equivalent to
$\Delta=\lambda_{\pm}=0$.

When (\ref{eq:E7_cubic}) has one root then the pencil has one singular curve in
the corresponding chart. By the classification of singular fibers in elliptic 
fibrations, the unique singular fiber with a single singular point 
besides an $\Et7$-fiber must be of type $III$.

Conversely, if the fibration has a type $III$ fiber, then the pencil has no other singular point. 
This requires that the cubic in (\ref{eq:E7_cubic}) has only one root, which 
is equivalent to $\Delta=\lambda_{\pm}=0$.
Hence we get part~(\ref{thm:mainE7+III}) of Theorem~\ref{thm:mainE7}.

\subsubsection{Two roots}
At the same time we verified part~(\ref{thm:mainE7+II+I1}) as well,
since $\Delta=0$ and ${\Delta_0}\neq 0$ is equivalent to $\Delta=0$ and
$\lambda_{\pm}\neq 0$, furthermore Equation~(\ref{eq:E7_cubic}) has two
distinct roots. By the classification the only possibility is the fibration 
has fibers of types $II$ and $I_1$.

\subsubsection{Three roots}
Now, we consider the $\Delta \neq 0$ case, where the cubic in (\ref{eq:E7_cubic}) 
has three distinct roots and the fibration has three singularities. 

\begin{lem}\label{lem:equivalence}
$\Delta \neq 0$ and $\lambda_{\pm}=0$ holds if and only if the pencil
  has $I_2$ and $I_1$ singularities.
\end{lem}

\begin{proof}
The first direction is simple, because the equation of
{(\ref{eq:E7_cubic})} can be easily solved and the three $z_2$ values can
be substituted into Equation~(\ref{eq:E7_t}). We get two distinct $t$ values,
which correspond to one curve with two singularities and another curve
with one singularity. The three $z_2$ values are:
\begin{align*}
	(z_2)_1 &= \frac{b_+-b_-}{2 \left(a_--a_+\right)}, \\
	(z_2)_{2,3} &= \frac{b_+-b_- \pm \sqrt{\left(b_--b_+\right){}^2-4 \left(a_--a_+\right) \left(c_--c_+\right)}}{2 \left(a_--a_+\right)}.
\end{align*}
The two $t$ values are:
\begin{align*}
t_1 =& -\frac{1}{64 \left(a_--a_+\right){}^2} \left(16 \left(a_--a_+\right){}^2 c_-^2 + 8 \left(a_--a_+\right) c_- \cdot \right.\\
	&\left. \cdot \left(4 \left(a_--a_+\right) c_+-\left(b_--b_+\right){}^2\right) + \left(\left(b_--b_+\right)^2 + 4 \left(a_--a_+\right) c_+\right)^2 \right), \\
t_{2,3} =& -c_- c_+.
\end{align*}

Let us see now the converse direction. If the pencil contains an $I_2$
and an $I_1$ curve then the equation of (\ref{eq:E7_cubic}) has three
distinct roots. Let us denote these roots by $y_1, y_2,
y_3$. {Denote the value of $t$ by $t_i$ after the substitution
  of $z_2$ with $y_i$ in Equation~(\ref{eq:E7_t}).}  Two roots (say
$y_1$ and $y_2$) provide singularities on the same curve, that is,
{$t_1=t_2$}. Equivalently
\begin{equation*}
\begin{split}
	0={t_1-t_2}=& 4 \left(a_--a_+\right) \lambda _+-\left(\left(a_--a_+\right) \left(y_1+y_2\right)+b_--b_+\right) \cdot\\
	&\cdot \left(\left(a_--a_+\right) \left(y_1^2+y_2^2\right)+ \left(b_- -b_+\right) \left(y_1+y_2\right)+2 \left(c_-- c_+\right)\right),
\end{split}
\end{equation*}
where we simplify with $\frac{1}{4}(y_1-y_2)$.

Obviously, the three distinct roots provide two values for $t$ if and only if
\begin{equation*}
	0= \left( t_1-t_2 \right)\left( t_2-t_3 \right)\left( t_3-t_1 \right).
\end{equation*}

This expression is a symmetric polynomial in $y_1, y_2, y_3$, hence
can be written as a polynomial of the elementary symmetric polynomials
${\sigma_1}=y_1+y_2+y_3$, ${\sigma_2}=y_1y_2+y_2y_3+y_3y_1$ and ${\sigma_3}=y_1y_2y_3$.  
The above vanishing condition would yield a long expression, but {$\sigma_1,
\sigma_2, \sigma_3$} can be determined from the coefficient of equation
(\ref{eq:E7_cubic}) by Vieta's formulas. 
The relations between the symmetric polynomials and
the coefficients then provide
\begin{align*}
	\sigma_1 =& -\frac{3 \left(b_--b_+\right)}{2 \left(a_--a_+\right)},\\
	\sigma_2=& \frac{2 \left(a_--a_+\right) \left(c_--c_+\right)+\left(b_--b_+\right){}^2}{2 \left(a_--a_+\right){}^2},\\
	\sigma_3 =& \frac{2 \left(a_--a_+\right) \lambda _+-\left(b_--b_+\right) \left(c_--c_+\right)}{2 \left(a_--a_+\right){}^2}.
\end{align*}

After simplifications, we get a condition for fibration to contain an
$I_2$ curve:
\begin{equation}
\label{eq:T}
	 \frac{\lambda _+ \left(\left(\left(b_--b_+\right){}^2-4
           \left(a_--a_+\right) \left(c_--c_+\right)\right){}^3-432
           \left(a_--a_+\right){}^4 \lambda _+^2\right)}{16
           \left(a_--a_+\right)} =0.
\end{equation}
Now $\Delta$ is in the nominator, and $\lambda_+$ is a multiplication
factor, hence Equation~\eqref{eq:T} becomes the following:
\begin{equation*}
	-\frac{\Delta  \lambda _+}{16 \left(a_--a_+\right)}=0.
\end{equation*}
Since the pencil has three singularities, we have that $\Delta
\neq 0$, consequently $\lambda_+ =0$ concluding the proof of
Lemma~\ref{lem:equivalence}.
\end{proof}

\begin{rk}
Note that according to Proposition~\ref{prop:DoublesectionAnd2section}
the fibration has the fiber of type $III$ or $I_2+I_1$ if and only if
the pencil contains a double section which is the union of two
sections, and this last condition is easily seen to be equivalent to
$\lambda_{\pm}=0$.
\end{rk}

We need to examine the last case when $\Delta \neq 0$ and
$\lambda_{\pm} \neq 0$.  {By process of elimination} there is a
single possibility for the singular fibers: there are three $I_1$
fibers in the fibration.

In summary, so far we have identified all possible singular curves in
the pencil on $X$. We summarize the cases in Table \ref{tab:E7}.
\begin{table}[htb]
\begin{center}
\begin{tabular}{c|c|c}
 & $\lambda_+ = 0$ & $\lambda_+ \neq 0$\\ 
\hline 
$\Delta=0$ & $III$ &  $II+I_1$\\ 
\hline 
$\Delta\neq 0$ & $I_2+I_1$ & $3 I_1$ \\
\end{tabular}
\end{center}
\caption{\quad The type of singular curves in untwisted case}
\label{tab:E7}
\end{table}

By Lemma~\ref{lem:ag} the same classification applies for the fibers of
$p\circ \sigma \colon \Xtt \to \CP1$.  The fibration obtained from the
pencil has a section (actually, even two sections).  Let $\vert D\vert
_0^{sm}$ be the subset of $\vert D\vert _0$ parametrizing smooth
curves.  The relative Abel--Jacobi map gives an algebraic isomorphism
between the restriction of the fibration to $\vert D\vert _0^{sm}$ and
its relative Picard scheme.  Therefore, in order to conclude the proof
of Theorem~\ref{thm:mainE7}, one merely needs to study torsion-free
sheaves on the singular fibers of $H$.  In the cases of curves of
types $I_1 , II$ this was carried out in Proposition
\ref{prop:compactified-Picard-scheme-I1-II}.  For curves of types
$I_2$ and $III$, the analysis is carried out in
Section~\ref{sec:stability-analysis}.

\section{Stability analysis in the untwisted case}
\label{sec:stability-analysis}

 In cases \eqref{thm:mainE7+II+I1} and \eqref{thm:mainE7+I1+I1+I1} of
Theorem~\ref{thm:mainE7}
the singular fibers of the elliptic pencil 
 (except for the type $\Et7$ fiber at infinity) are integral (i.e. irreducible and reduced), 
 so the Hitchin fiber of the moduli space corresponding to the singular fibers is just the usual 
 compactified Picard scheme of degree $\delta$. 
 In the other cases however we need to determine the Hitchin fibers of $\Mod$ corresponding to the reducible 
 singular fibers of the pencil. 

\subsection{Stability analysis in the case \texorpdfstring{$\Et7 + I_2 + I_1$}{E\textseveninferior+I\texttwoinferior+I\textoneinferior}}
We use the results and notations of Subsection \ref{subsubsec:Oda-Seshadri-I2}. 
We let $b\in B$ denote the point whose preimage in the pencil is the singular fiber of type $I_2$. 
We assume that 
$$
  \E = p_* (\Ft ) 
$$
for some torsion-free sheaf $\Ft$ of $\O_{X_b}$-modules of rank $1$ and use the definitions (\ref{eq:delta-i}). 
By assumption we have 
\begin{equation}\label{eq:0=pardeg}
  0 = \pardeg (\E) = \deg (\E ) + \alpha_0^+ + \alpha_0^-.
\end{equation}
Then in view of (\ref{eq:delta=delta1+delta2+2-J}) and  (\ref{eq:delta-d}) the above formula may be rewritten as 
\begin{equation}\label{eq:pardeg=0-delta}
   0 = (\delta_+ + \alpha_0^+ ) + (\delta_- + \alpha_0^- ) - |J(\Ft )| .
\end{equation}
For any non-trivial Higgs subbundle $(\F, \t|_{\F})$ of $(\E, \theta)$ the scheme 
$$
  (\t|_{\F} - \lambda ) \subset X 
$$
is a sub-scheme of $X_b$ that is a one-to-one cover of $\CP1$. 
Clearly the same also holds for non-trivial quotient Higgs bundles 
$(\Qt, {\overline {\t}} )$. 
On the other hand, for any $i \in \{ \pm \}$ the functor $p_*$ applied to the morphism (\ref{eq:canonical-quotient}) 
gives rise to a quotient Higgs bundle $(\Qt_i, {\overline {\t}} )$. 
Again by (\ref{eq:delta-d}) the degree of this quotient is given by $\delta_i$ so its parabolic degree is 
$$
  \delta_i + \alpha_0^i . 
$$
It is easy to see that these are the only quotient Higgs bundles of $(\E, \theta)$, 
because the support of the spectral sheaf of any such quotient is a component of $X_b$, 
and there are exactly two such components. 
We infer that $(\E, \theta)$ is $\vec{\alpha}$-stable if and only if the two inequalities 
$$
  \delta_i + \alpha_0^i > 0 
$$
for $i \in \{ \pm \}$ hold. 
Taking into account the formula (\ref{eq:pardeg=0-delta}) these inequalities are also equivalent to 
\begin{equation}\label{eq:parabolic-stability-delta}
  0 < \delta_i + \alpha_0^i < |J(\Ft )|  
\end{equation}
for $i \in \{ \pm \}$. 
Let us point out that this can only have a solution if $|J(\Ft )| \in  \{ 1, 2 \}$. Now, setting 
$$
  \phi_i = 1 - \alpha_0^i ,
$$
we see that the stability condition (\ref{eq:parabolic-stability-delta}) transforms into 
(\ref{eq:stability-I2}), which is the Oda--Seshadri stability condition for the values $(\phi_+,\phi_-)$. 
(Notice however that the equality 
$$
  \phi_- + \phi_+ = 0
$$
holds if and only if 
$$
   \alpha_0^+ + \alpha_0^- = 2,
$$
which is incompatible with our assumption that $0 \leq \alpha_0^{\pm} < 1$.)

Let us explicitly write down the corresponding Hitchin fibers. 
For simplicity let us set 
$$
   \alpha^i = \alpha_0^i . 
$$
Since 
$$
  \alpha^+ + \alpha^-= - \deg (\E) 
$$
is an integer and $\alpha^+, \alpha^- \in [0,1)$, it follows that 
\begin{itemize}
 \item either we have $\deg(\E) = -1$ and
 \begin{equation}\label{eq:alphasumto1}
    \alpha^+ + \alpha^-= 1
  \end{equation}
 \item or we have $\deg(\E) = 0$ and
 \begin{equation}\label{eq:alphasumto0}
    \alpha^+ = 0 =  \alpha^-.
  \end{equation}
\end{itemize}

\subsubsection{Case of degree $-1$} 

Assume that $d = \deg(\E) = -1$. By virtue of (\ref{eq:delta-d}) this
amounts to $\delta = \deg (\Ft ) = 1$.  Let us first study the sheaves
with $|J(\Ft )| = 2$, i.e. invertible sheaves on $X_b$.  Assumptions
(\ref{eq:alphasumto1}) and $\alpha_i\in [0,1)$ imply that $\alpha_i\in
  (0,1)$, therefore, by condition
  (\ref{eq:parabolic-stability-delta}) we have either
$$
  \delta_+ = 0, \delta_- = 1
$$
or 
$$
  \delta_+ = 1, \delta_- = 0.
$$
Let us introduce the notation 
$$
  \L_i = \L(\Ft)|_{X_i}
$$
and fix one of the two conditions on the degrees spelled out above. 
Then, as $X_{\pm}$ are rational curves, the isomorphism class of $\L_{\pm}$ is completely determined. 
Moreover, according to (\ref{eq:OS-SES}), $\Ft$ is obtained by first identifying the fibers $(\L_+)_0$ and $(\L_-)_0$ 
by an isomorphism, then identifying the fibers $(\L_+)_{\infty}$ and $(\L_-)_{\infty}$ by an isomorphism. 
The possible identifications between these pairs of lines are parametrized by $\C^{\times} \times \C^{\times}$. 
Indeed, for trivializations $\sigma_i$ over open affine subsets of $X_i$, we have 
$$
  \sigma_+(0) = \lambda_0 \sigma_-(0), \quad \sigma_+(\infty ) = \lambda_{\infty} \sigma_-(\infty )
$$
for some 
$$
  (\lambda_0, \lambda_{\infty}) \in \C^{\times} \times \C^{\times} \subset \C^2. 
$$
However, we may act on this space of identifications by constant automorphisms of one of the bundles $\L_i$ (say $\L_+$) 
without changing the isomorphism class of the sheaf $\Ft$ obtained by the identifications. 
Constant automorphisms are isomorphic to $\C^{\times}$ and $t\in \C^{\times}$ obviously acts by 
$$
  t (\lambda_0, \lambda_{\infty}) = (t \lambda_0, t \lambda_{\infty}). 
$$
Therefore, we are left with a parameter space 
$$ \mbox{Pic}^{\delta_+, \delta_-} = \C^{\times} \times \C^{\times} /
\C^{\times} = \C^{\times} \subset \CP1 $$ for such invertible
sheaves. It is easy to see that these sheaves  are all
non-isomorphic.  This implies that the universal line bundle on $X_b$
of bidegree $(\delta_+, \delta_-)$ is given by
\begin{equation*}
   L^{\delta_+, \delta_-}(\cdot ) \to \mbox{Pic}^{\delta_+, \delta_-} \times X_b = \C^{\times} \times X_b . 
\end{equation*}
Now let us consider the case of sheaves $\Ft$ with $|J(\Ft )| = 1$.
These sheaves are locally free in a neighborhood of exactly one of
the two points $\{ 0, \infty \}$.  Clearly, if $\Ft$ is locally free
near $0$ and not locally free near $\infty$ then $\Ft$ cannot be
isomorphic to a sheaf $\Ft'$ that is locally free near $\infty$ and
not locally free near $0$.  Thus there exist at least $2$ points in
$$
  \cJ_{X_b}^{\delta,\phi} \setminus (\Pic^{0,1} \cup \Pic^{1,0}).
$$
Our aim is to show that there exist exactly $2$ points in this complement. 
Indeed, we first observe that if $|J(\Ft )| = 1$ then (\ref{eq:parabolic-stability-delta}) only allows for 
$$
    \delta_+ = 0 = \delta_- .
$$ As $X_{\pm}$ are rational curves, the isomorphism class of line
    bundles of degree $0$ on $X_{\pm}$ is unique, they are given by
    $\L_{\pm} = \O_{X_{\pm}}$.  Now, assume that $\Ft$ is locally free
    near $0$.  Then $\Ft$ is obtained by identifying the fibers
    $(\L_+)_0$ and $(\L_-)_0$ by a linear isomorphism.  The choices
    for such an isomorphism are parametrized by $\C^{\times}$.
    However, we again get isomorphic sheaves if we apply a constant
    automorphism to one of $\L_{\pm}$.  It follows that there exists a
    single stable sheaf $\Ft_0$ that is locally free near $0$ but not
    locally free near $\infty$.  Similarly, there exists a unique
    stable sheaf $\Ft_{\infty}$ that is locally free near $\infty$ but
    not locally free near $0$.

Finally, we show that both $\Ft_0$ and $\Ft_{\infty}$ are in the closure of both $\Pic^{0,1}$ and $\Pic^{1,0}$ in 
$\cJ_{X_b}^{\delta,\phi}$. The argument closely follows the one in the proof of Proposition 
\ref{prop:compactified-Picard-scheme-I1-II}. 
Let us for instance work in the chart $\lambda_{\infty} = 1$ of $\CP1$, 
and fix one of the two conditions on the degrees spelled out above, say $(1,0)$. 
We will consider the limit $L^{1,0}(0)$ of the line bundles $L^{1,0}(\lambda )$ as $\lambda = \lambda_0 \to 0$. 
Let us denote the two preimages of $0\in X_b$ by $0_+ \in X_+, 0_-\in X_-$ respectively. 
For $\lambda_0 = 0$ we get 
$$
  \sigma_+(0_+) = 0 \cdot \sigma_-(0_-),
$$
hence 
$$
  \Tor^{\O_{X_+, 0_+}}(\pi^* L^{1,0}(0)) \cong \C_{0_+}
$$
is generated by $\sigma_-(0_-)$, and 
$$
  \Tor^{\O_{X_+, 0_+}}(\pi^* L^{1,0}(0)) \cong 0. 
$$
At the points $\infty_{\pm}$, $L^{1,0} (0)$ is locally free. 
We infer that the line bundle $\L_+(0)$ of (\ref{eq:associated-invertible-sheaf}) 
over $X_+$ associated to $L^{1,0} (0)$ fits into the short exact sequence 
$$
  0 \to \L_+(0) \to \O_{X_+} (1) \to \C_{0_+} \to 0, 
$$ and that $\L_-(0) = \O_{X_-}$; in other words, these line bundles
  are both of degree $0$.  As we have already shown, $\Ft_{\infty}$ is
  up to isomorphism the unique sheaf of bidegree $(0,0)$ which is
  locally free near $\infty$ but not locally free near $0$. We infer
  that
$$
   L^{1,0} (0) = \Ft_{\infty}. 
$$
A similar argument for $L^{0,1}$ over the affine chart $\lambda_{\infty} = 1$ now shows that 
the limit of $L^{0,1} (\lambda )$ as $\lambda\to 0$ is a sheaf of bidegree $(-1,1)$, 
locally free near $\infty$ but not locally free near $0$. 
Let us denote by $X_0$ the partial normalization of $X_b$ at the point $0\in X_b$. 
By the uniqueness of $\Ft_{\infty}$ we see that 
$$
  L^{0,1} (0 ) \cong \Ft_{\infty} \otimes_{\O_{X_0}} \O_{X_0} (- \{ 0_+ \} + \{ 0_- \}). 
$$
However, as the arithmetic genus of $X_0$ is $0$, the latter sheaf is trivial. 
Hence, $L^{0,1} (0 )$ is also isomorphic to  $\Ft_{\infty}$. 

The case of $\Ft_0$ can then be obtained by exchanging the roles of $0$ and $\infty$. 

We infer from the discussion above that the moduli space has the structure of an elliptic fibration near the 
point $b \in B$ corresponding to the singular fiber. 
Furthermore, it is easy to check (using the fact that the parabolic weights are non-zero) that in this case semi-stability is equivalent to stability. 
Therefore, by \cite{Biq-Boa} the moduli space is complete. 
It then follows that the fiber of the Hitchin map $H$ over $b$ is either a smooth elliptic curve or one of the singular fibers on Kodaira's list. 
As we have shown above, this fiber is homeomorphic to two copies of $\CP1$ attached at two different points. 
In particular, the fiber is singular, and as the only fiber on Kodaira's list homeomorphic to two copies of $\CP1$ attached at two points is $I_2$,  
we conclude that $H^{-1}(b)$ is a type $I_2$ curve.

\subsubsection{Case of degree $0$}

The analysis is similar to the case of degree $-1$, hence we only give the outline. 
In the case $|J(\Ft )| =2$ of invertible sheaves, we obtain 
$$
   \delta_+ = 1 = \delta_- ,
$$
and if $|J(\Ft )| =1$ then no $(\delta_+ ,\delta_- )$ solves (\ref{eq:parabolic-stability-delta}). 
We infer that stable sheaves are parametrized by $\C^{\times}$. 
Let us now consider strictly semi-stable sheaves. 
Then, the solutions in the case $|J(\Ft )| =2$ are 
$$
   \delta_+ \in \{ 0,1,2 \},
$$
with $\delta_- = 2 -\delta_+$. The parameter space consists of $3$ copies of $\C^{\times}$. 
The solutions $(\delta_+ ,\delta_- )$ with $|J(\Ft )| =1$ are 
$$
  (0,1), \quad (1,0), 
$$
each being parametrized by a point. 
The point corresponding to bidegree $(0,1)$ is both a limit point of the $\C^{\times}$ 
parametrizing invertible sheaves of bidegree $(0,2)$ and the one parametrizing 
invertible sheaves of bidegree $(1,1)$. 
Similarly, the point corresponding to bidegree $(1,0)$ is both a limit point of the $\C^{\times}$ 
parametrizing invertible sheaves of bidegree $(2,0)$ and the one parametrizing 
invertible sheaves of bidegree $(1,1)$. 
All the semi-stable solutions are parametrized by a copy of $\CP1$ with two copies of $\C$ attached 
to it at two different points of $\CP1$. 
In contrast with the case of degree $-1$, this time there do exist strictly semi-stable Higgs bundles, 
and in addition the parabolic weights are not all distinct. 
Hence, we cannot use a completeness argument to determine the algebraic type of the singular fiber. 

\subsection{Stability analysis in the case \texorpdfstring{$\Et7 + III$}{E\textseveninferior+III}}
\label{subsec:III}
We now let $b\in B$ be the point whose preimage in the pencil is the singular fiber of type $III$. 
We again have (\ref{eq:0=pardeg}). 

We assume that 
$$
  \E = p_* (\Ft ) 
$$
and use the definitions of (\ref{eq:delta-i}). 
The curve $X_b$ has a single singular point $x$ which is a tacnode (an $A_3$-singularity). 
It is known that there exists a fractional ideal 
$$
  \O_{X_b,x} \subseteq I \subseteq \O_{\tilde{X}_b,x}
$$
of $\O_{X_b,x}$ such that 
$$
  \Ft_x \cong I. 
$$
The length of $\Ft$ at $x$ is by definition 
$$
  l(\Ft) = \dim_{\C}(I /\O_{X_b,x}), 
$$
and we have the inequalities  
$$
  0 \leq l(\Ft) \leq \dim_{\C}(\O_{\tilde{X}_b,x}/\O_{X_b,x}) = 2.
$$
Now there exists a short exact sequence of sheaves 
$$
  0 \to \Ft \to \L(\Ft)|_{X_+} \oplus \L(\Ft)|_{X_-} \to \C_x^{2-l(\Ft)}\to 0, 
$$
hence 
$$
  \chi (\Ft) + 2 - l(\Ft) = \chi (\L(\Ft)|_{X_+}) + \chi (\L(\Ft)|_{X_-}). 
$$
Applying this to $\O_{X_b}$ in the place of $\Ft$ we get 
$$
  \chi (\O_{X_b}) + 2 = \chi (\O_{X_+}) + \chi (\O_{X_-}).  
$$
Subtracting the second formula from the first we infer 
$$
  \delta - l(\Ft) = \delta_+ + \delta_-,
$$ with $\delta,\delta_+, \delta_-$ the degrees of $\Ft,
  \L(\Ft)|_{X_+}$ and $ \L(\Ft)|_{X_-}$, respectively.  Using this
  formula and (\ref{eq:delta-d}) we can rewrite (\ref{eq:0=pardeg}) as
\begin{equation}\label{eq:0=pardeg-delta}
   0 = \delta_+ + \delta_- + l(\Ft) - 2 + \alpha^+ + \alpha^-.
\end{equation}
The canonical morphisms (\ref{eq:canonical-quotient}) give quotient irregular parabolic Higgs bundles 
$\E_{i}$ of $\E$ of rank $1$ and degree 
$$
  d_{i} = \delta_{i}
$$
for $i \in \{ \pm \}$. 
Furthermore, these are again the only non-trivial Higgs quotient bundles of $\E$. 
The parabolic weight associated to $\E_i$ is $\alpha^i$, 
so the parabolic degree of $\E_{i}$ is 
$$
  \pardeg (\E_i) =  \delta_{i} + \alpha^i . 
$$
It follows that the parabolic stability of $(\E,\theta)$ is equivalent to the inequalities 
$$
  0 < \delta_{i} + \alpha^i
$$ for $i \in \{ \pm \}$. Taking (\ref{eq:0=pardeg-delta}) into
  account, this is equivalent to
\begin{equation}\label{eq:stability-III}
   \delta_{+} + \alpha^+ + 2 l(\Ft) - 2 < \delta_{-} + \alpha^- + l(\Ft) < \delta_{+} + \alpha^+ + 2. 
\end{equation}
This time this inequality immediately implies that there exist no stable Higgs bundles with spectral sheaf $\Ft$ of length $2$. 

We again set 
$$
   \alpha^i = \alpha_0^i 
$$
and we need to distinguish two cases: 
\begin{itemize}
 \item either we have $\deg(\E) = -1$ and
 \begin{equation}\label{eq:alphasumto1new}
    \alpha^+ + \alpha^-= 1
  \end{equation}
 \item or we have $\deg(\E) = 0$ and
 \begin{equation}\label{eq:alphasumto0new}
    \alpha^+ = 0 =  \alpha^-.
  \end{equation}
\end{itemize}

\subsubsection{Case of degree $-1$}

Let us first treat the case of (\ref{eq:alphasumto1new}). 
Assume first $l(\Ft) = 0$, i.e. $\Ft$ is an invertible sheaf on $X_b$. 
Then, independently of the values of $\alpha^{\pm}$ satisfying (\ref{eq:alphasumto1new}), condition (\ref{eq:stability-III}) 
implies either 
$$
  \delta_+ = 0, \delta_- = 1 
$$
or 
$$
  \delta_+ = 1, \delta_- = 0. 
$$
Therefore, such sheaves are parametrized by $\C\coprod \C$, as it readily follows from the long exact sequence associated to 
$$
  0 \to \O_{X_b} \to \O_{\tilde{X}_b} \to \O_{\tilde{X}_b,x}/\O_{X_b,x} \to 0 
$$
using the fact that $\tilde{X}_b$ has two connected components. 

If, on the other hand, we have $l(\Ft) = 1$ then the only solution is 
$$
  \delta_+ = 0 = \delta_- ,
$$
again independently of the values of $\alpha^{\pm}$. 
This latter sheaf is in the closure of both components $\C$ parametrizing invertible sheaves. 
We infer that up to homeomorphism, the Hitchin fiber over the point $b$ is parametrized by two copies of 
$\CP1$ attached at one point. 
As the generic fiber of the Hitchin-fibration is an elliptic curve and the moduli space is complete by \cite{Biq-Boa}, 
the fiber over $b$ must be again one of the fibers of Kodaira's list. 
However, the only singular fiber on the list that is homeomorphic to two copies of $\CP1$ 
glued at one point is the fiber of type $III$. 
Therefore, the Hitchin fiber $H^{-1}(b)$ is a singular curve of type $III$. 

\subsubsection{Case of degree $0$}

Let us now study the case of (\ref{eq:alphasumto0new}): in this case,
by virtue of (\ref{eq:0=pardeg-delta}) we have
$$
  2 - l(\Ft) = \delta_+ + \delta_-.
$$ If $l(\Ft) = 0$ then we readily see that the only solution to
  equation (\ref{eq:stability-III}) is
$$
  \delta_+ = 1 = \delta_- ,
$$
and just as above one can show that such sheaves are parametrized by $\C$.

On the other hand, if $l(\Ft) = 1$ then (\ref{eq:stability-III}) has no solutions; however, 
if we relax the inequalities in (\ref{eq:stability-III}) to not necessarily strict ones, 
then there exist two solutions: 
\begin{equation}\label{eq:delta+0}
   \delta_+ = 0, \delta_- = 1 
\end{equation}
and
\begin{equation}\label{eq:delta-0}
\delta_+ = 1, \delta_- = 0.
\end{equation}
The sheaves with these properties are parametrized by one point in each of the two cases. 

Let us first analyze the case (\ref{eq:delta+0}): in this case, the destabilizing quotient of $(\E, \theta)$ 
is $\E_+$: indeed, we have 
$$
  0 = \delta_+ = \deg (\E_+ ) = \pardeg (\E_+ ) =  \deg (\E ) = \pardeg (\E ),
$$
since the parabolic weights vanish. The destabilizing Higgs subbundle of $\E$ is 
$$
  \ker (\E \to \E_+ ), 
$$
which is a lower elementary transformation of $\E_-$: 
$$
  \ker (\E \to \E_+ ) = \E_- (- \{ t \}), 
$$
where $t \in \CP1$ is the image under $p$ of the singular point of $X_b$. 
Indeed, we have 
\begin{equation}\label{eq:lower-elem-tr-degree}
   \deg (\E_- (- \{ t \})) = \deg (\E_- ) -1 = 0,
\end{equation}
and $\E_- (- \{ t \})$ is preserved by $\theta$ simply because the image by $\theta$ of 
vanishing sections of $\E$ at $t$ also vanish at $t$, in particular, they belong to $\E_-$. 
The Jordan--H\"older filtration of $(\E , \theta)$ is therefore given by 
$$
   \E_- (- \{ t \}) \subset \E ,
$$
with associated graded 
$$
  \E_- (- \{ t \}) \oplus  \E_+
$$
endowed with the action 
\begin{equation}\label{eq:JH-graded-Higgs-field}
   \begin{pmatrix}
   \theta_- & 0 \\
   0 & \theta_+
  \end{pmatrix},
\end{equation}
where $\theta_{\pm}$ are the morphisms induced by $\theta$ on the two direct summands. 
According to (\ref{eq:lower-elem-tr-degree}), the vector bundle underlying this graded Higgs bundle 
is isomorphic to the trivial bundle of rank $2$ over $\CP1$. 
Moreover, the action of $\theta_{\pm}$ in the above matrix clearly has spectral curve $X_{\pm}$ respectively. 

The case of (\ref{eq:delta+0}) can be treated in a very similar
manner, except that one needs to exchange the roles of $\E_+$ and
$\E_-$. It then follows that the destabilizing Higgs subbundle of $\E$
is
$$
  \E_+ (- \{ t \}), 
$$
and that the graded Higgs bundle associated to the Jordan--H\"older filtration is 
$$
  \E_-  \oplus  \E_+ (- \{ t \}), 
$$
the trivial vector bundle of rank $2$ over $\CP1$, 
with Higgs field given by the formula (\ref{eq:JH-graded-Higgs-field}). 

The upshot is that in both cases (\ref{eq:delta+0}) and (\ref{eq:delta-0}), 
the associated graded Higgs bundles for the Jordan--H\"older filtration 
have isomorphic underlying vector bundles, and the Higgs-field splits as a direct sum. 
Moreover, the spectral curves of $\theta_+$ are equal in both cases, and the same 
holds for $\theta_-$. We infer that the associated graded Higgs bundles of the Higgs 
bundles coming from spectral sheaves satisfying (\ref{eq:delta+0}) and (\ref{eq:delta-0}) 
are isomorphic. 
Said differently, the Higgs bundles associated to (\ref{eq:delta+0}) and (\ref{eq:delta-0}) 
are $S$-equivalent, therefore they are represented by the same point in $\Mod$. 

To sum up, in the degree $0$ case the Hitchin fiber over the point $b$
is homeomorphic to the compactification of $\C$ (corresponding to
invertible sheaves) by a unique point (corresponding to sheaves of
length $1$).  However, the parabolic weights are equal and there exist
strictly semi-stable Higgs bundles, so we cannot use a completeness
argument to determine algebraically the special fiber of the Hitchin
map.

\subsection{The proof of Theorem~\ref{thm:mainE7}}
Now we are in a position of proving our first main result.
\begin{proof}[Proof of Theorem \ref{thm:mainE7}]
The polar part of an irregular Higgs bundle depends on the
  parameters listed in (\ref{eq:untwisted-parameters}). The
  case-analysis in Subsection~\ref{subsec:E7_local} describes all
  possible singular fibers in the Hirzebruch surface $X$. The blow-up
  procedure in Subsection~\ref{subsec:RefBNR} and Lemma~\ref{lem:ag}
  provide a biholomorphism between $X$ and the constructed rational
  surface $\Xtt$. Proposition~\ref{prop:E7} guarantees the existence
  of an elliptic fibration on $\Xtt$ with an $\Et7$ singular fiber and
  describes the moduli space $\Mod^{ss}$.  Finally,
  Proposition~\ref{prop:compactified-Picard-scheme-I1-II} and the
  analysis in Section~\ref{sec:stability-analysis} identify the
  Hitchin fibers in $\Mod^{s}$ and hence verify
  Theorem~\ref{thm:mainE7}.
\end{proof}

\section{The twisted case}\label{sec:E8}

In this section we determine a certain blow-up $\Xtt$ of $\Xt$
depending on the parameters appearing in (\ref{eq:twisted-parameters})
with the property that certain sheaves on $\Xtt$ are in one-to-one
correspondence with Higgs bundles of the local form
(\ref{eq:local-form-E8}).  We need two preliminary lemmas.

\begin{lem}\label{lem:an-bn}
Let $\t$ be a Higgs field of the local form (\ref{eq:local-form-E8}). 
Let us denote by $\zeta \d z$ the eigenvalues of $\t$; $\zeta$ is a ramified bi-valued meromorphic function of $z_1$. 
\begin{enumerate}
 \item Assume that $b_{-7} \neq 0$. Then, for $-8 \leq n \leq -3$ the coefficients of the Puiseux expansion 
\begin{equation}\label{eq:Puiseux-E8}
  \zeta = \sum_{n=-8}^{\infty} a_n z_1^{\frac n2}. 
\end{equation}
admit expressions 
$$
  a_n = a_n(b_{-8}, \sqrt{b_{-7}}, b_{-6},\ldots , b_n) \in \C [ b_{-8}, b_{-7}^{\pm 1/2}, b_{-6},\ldots , b_n ]
$$
in the parameters $b_n$, and $a_{-7}\neq 0$. \label{lem:an-bn-1}
\item Vice versa, if $\t$ is of the local form (\ref{eq:local-form-E8}) and $a_{-7}\neq 0$ then the parameters $b_{-8}, \ldots , b_{-3}$ admit 
polynomial expressions 
$$
  b_n = b_n(a_{-8},\ldots , a_n) \in \C [a_{-8},\ldots , a_n]
$$
in function of the Puiseux coefficients of $\zeta$, and $b_{-7} \neq 0$.\label{lem:an-bn-2}
\end{enumerate}
\end{lem}

\begin{proof}
 This is a straightforward computation. Specifically, we have 
 $$
  a_{-8} = b_{-8}, 
 $$
 for  $n\in \{ -6, -4 \}$ we have 
 $$
  a_n = \frac{b_n}2,
 $$
 and the coefficients with odd indices are given by 
 \begin{align*}
  a_{-7} & = \sqrt{b_{-7}}, \\
  a_{-5} & = \frac 1{8\sqrt{b_{-7}}} (b_{-6}^2 + 4b_{-5}), \\
  a_{-3} & = \frac 1{8\sqrt{b_{-7}}} (2b_{-4}b_{-6} + 4b_{-3}) - \frac 1{128b_{-7}\sqrt{b_{-7}}} (b_{-6}^2 + 4b_{-5})^2,
 \end{align*}
 (the square root of $b_{-7}$ depending on the choice of square root of $z$ in the Puiseux series).
  The inverse transformations are given by 
 \begin{align*}
    b_{-7} & = a_{-7}^2 \\
    b_{-5} & = 2 a_{-5}a_{-7} -  a_{-6}^2 \\
    b_{-3} & = 2 a_{-3}a_{-7} - 2 a_{-4}a_{-6} + a_{-5}^2.
 \end{align*}
\end{proof}

In the lemma below we follow the conventions and notations introduced in Sections \ref{sec:prep-mat} and \ref{sec:E7}. 
In particular, in view of the definition of the affine coordinate system $(z_1,w_1)$ near $p^{-1}(q) \setminus C^{\infty}$ 
and Lemma \ref{lem:an-bn}, the equation of the spectral curve of a Higgs field of the local form (\ref{eq:local-form-E8}) reads as 
$$
  w_1 = \sum_{n=0}^{\infty} a_{n-8} z_1^{\frac n2}. 
$$

\begin{lem}\label{lem:dn-an}
Assume the above Puiseux expansion holds. 
\begin{enumerate}
 \item If $a_{-7} \neq 0$, then for $2\leq n \leq 6$ there exist polynomials 
$$
  d_n = d_n(a_{-7},\ldots , a_{n-9} ) \in \C[a_{-7}^{\pm 1}, a_{-6},\ldots , a_{n-9}]
$$
such that we have the Taylor series \label{lem:dn-an-1}
\begin{equation}\label{eq:Taylor-E8}
   z_1 = d_2 (w_1- a_{-8})^2 + \cdots + d_6 (w_1- a_{-8})^6 + O((w_1- a_{-8})^7). 
\end{equation}
Moreover, $d_2 \neq 0$. 
\item Conversely, the value $a_{n-9}$ is a polynomial in $d_2^{\pm 1/2}, d_3,\ldots ,d_n$, and $a_{-7} \neq 0$. \label{lem:dn-an-2}
\end{enumerate}
\end{lem}

\begin{proof}
 By assumption we have 
 $$
    \frac{w_1- a_{-8}}{a_{-7}} = \sum_{n=1}^{\infty} \frac{a_{n-8}}{a_{-7}} z_1^{\frac n2} = z_1^{\frac 12} + O(z_1). 
 $$
 By formally inverting this series and then squaring the result we obtain the first claim. 
 In concrete terms we find  
\begin{align*}
  d_1 & = 0 \\
  d_2 & = \frac 1{a_{-7}^2} \\
  d_3 & = -2 \frac{a_{-6}}{a_{-7}^4} \\
  d_4 & = \frac{5a_{-6}^2 - 2 a_{-5}a_{-7}}{a_{-7}^6} \\
  d_5 & = \frac{-14a_{-6}^3 + 12 a_{-5}a_{-6}a_{-7} - 2 a_{-4}a_{-7}^2 }{a_{-7}^8} \\
	d_6 & = \frac{42 a_{-6}^4-56 a_{-7} a_{-5} a_{-6}^2+14 a_{-7}^2 a_{-4} a_{-6}+7 a_{-7}^2 a_{-5}^2-2 a_{-7}^3 a_{-3}}{a_{-7}^{10}}.
\end{align*}
 The converse statement follows directly. 
\end{proof}

\begin{lem}\label{lem:dn-bn}
Assume that Lemmas~\ref{lem:an-bn} and \ref{lem:dn-an} hold.
\begin{enumerate}
 \item If $b_{-7} \neq 0$, then for $2\leq n \leq 6$ there exist polynomials 
$$
  d_n = d_n(b_{-7},\ldots , b_{n-9} ) \in \C[b_{-7}^{\pm 1}, b_{-6},\ldots , b_{n-9}]
$$
such that we have the Taylor series
\begin{equation*}
   z_1 = d_2 (w_1- b_{-8})^2 + \cdots + d_6 (w_1- b_{-8})^6 + O((w_1- b_{-8})^7). 
\end{equation*}
Moreover, $d_2 \neq 0$. 
\item Conversely, the value $b_{n-9}$ is a polynomial in $d_2^{\pm 1}, d_3,\ldots ,d_n$, and $b_{-7} \neq 0$.
\end{enumerate}
\end{lem}

\begin{proof}
The lemma directly follows from the previous two lemmas.
\end{proof}

We now proceed to construct the surface $\Xtt$ with a birational morphism to $\Xt$ whose geometry governs $\Mod$. 
The idea is similar to the untwisted case: we use the above expansions to recursively find the point on the exceptional divisor that we blow up in the following step.
We assume that $\sigma_1 : \Xt \to X$ is the blow-up of $X$ in the point 
\begin{equation}\label{eq:a-8_kappa}
   [a_{-8} \kappa _1(0) : \mathbf{1} ]. 
\end{equation}
Let $E_1 \subset \Xt$ denote the corresponding exceptional divisor, see Figure \ref{fig:conf}. 
Observe that the coordinates $[z_1' : w_1']$ on $E_1$ now satisfy 
$$
  \frac{z_1'}{w_1'} = \frac{z_1}{w_1 - a_{-8}}, 
$$ so on a curve $\widetilde{\Sigma}$, having the expansion of
  (\ref{eq:Taylor-E8}), we have
$$
  \frac{z_1'}{w_1'} = \sum_{n=2}^{\infty} d_n (w_1- a_{-8})^{n-1}.
$$
We define 
$$
  \sigma_2 : X_2 \to \widetilde{X}
$$
as the blow-up of the point 
$$
  [z_1' : w_1'] = [0 : 1] \in E_1. 
$$
In concrete terms, on the affine chart $V_1 = \{ w_1' \neq 0 \} \subset \widetilde{X}$ we normalize $w_1'=1$ and in the 
affine coordinates $(z_1', w_1)$ on $V_1$ we consider 
$$
  \{ (z_1', w_1, [z_1'' : w_1'']) \in V \times \CP1 | \quad w_1''z_1' - z_1''(w_1- a_{-8}) = 0  \}. 
$$
(Observe that we have met the exceptional divisor of $\sigma_2$ in Figure \ref{fig:conf} under the name $E_4$.) 
With these definitions, over $V_2 = \{ w_1'' \neq 0 \} \subset X_2$ on a curve $\widetilde{\Sigma}$ having the expansion of (\ref{eq:Taylor-E8}) we have 
\begin{align*}
  \frac{z_1''}{w_1''} & = \frac{z_1'}{w_1- a_{-8}} \\
	& = \frac{z_1 w_1'}{(w_1- a_{-8})^2} \\
	& = \frac{z_1}{(w_1- a_{-8})^2} \\
	& = \sum_{n=2}^{\infty} d_n (w_1- a_{-8})^{n-2}
\end{align*}
(recall we have set $w_1' = 1$).

{From} this point on, the pattern of the construction of $\Xtt$ is clear and similar to the construction in the untwisted case. 
Namely, for $3 \leq n \leq 8$ we successively consider the blow-up 
$$
    \sigma_{n} : X_{n} \to X_{n-1} 
$$
of the point 
$$
  [z_1^{(n-1)} : w_1^{(n-1)}] = [d_{n} : 1] \in E_{n+1}
$$
and denote by $E_{n+2}$ the exceptional divisor of $\sigma_{n}$. 
We set 
$$
  \Xtt = X_8 ,
$$
and define 
\begin{equation}\label{eq:sigma_twisted}
   \sigma = \sigma_8 \circ \cdots \circ \sigma_1 : \Xtt \to X. 
\end{equation}

\begin{prop}\label{prop:equivalence}
 There exists an equivalence of categories between the groupoids of 
 \begin{enumerate}
  \item Higgs bundles on $\CP1$ with one singular point $q = 0$ and
    local form given by (\ref{eq:local-form-E8}) with $b_{-7} \neq
    0$ \label{prop:equivalence1}, and 
  \item pure sheaves of dimension $1$ and rank $1$ on $\Xtt$ supported
    on a curve $\widetilde{\Sigma}$ which is disjoint from $E_1,\ldots
    ,E_9$ and intersects $E_{10}$ with algebraic multiplicity
    $1$.  \label{prop:equivalence2}
 \end{enumerate}
\end{prop}

\begin{proof}
 Let $(\E, \theta)$ be a Higgs-field as in part
 \eqref{prop:equivalence1}. Consider its spectral sheaf
 $$
    \Ft_0 = \coker \left( p^*(\E\otimes \Theta_C (-4 \cdot \{ 0 \}))  \xrightarrow{\xi\otimes  p^*\theta+\zeta} p^*(\E)\otimes \O_Z(1) \right),
 $$
 where $\Theta_C (-4 \cdot \{ 0 \})$ is the dual bundle of $K_C(4 \cdot \{ 0 \})$, 
 and $\xi \in H^0(Z,\O_Z(1)), \zeta\in H^0(Z, p^*(K_C(4 \cdot \{ 0 \}))\otimes \O_Z(1))$ are the canonical sections. 
 Let us denote by $\Sigma_0$ the support of $\Ft_0$. Assume that $\Sigma_0$ is integral (i.e. irreducible and reduced).
 Then, by \cite{BNR}, we have 
 \begin{itemize}
  \item $\Sigma_0$ is disjoint from $C_{\infty}$, 
  \item $p$ is finite over $\Sigma_0$, 
  \item $\Ft_0$ is torsion-free on $\Sigma_0$, 
  \item $p_* \Ft_0 = \E$, 
  \item the direct image of multiplication by $\zeta$ on $\Ft_0$ induces $\theta$. 
 \end{itemize}
 Conversely, any sheaf $\Ft_0$ satisfying the first three of these properties is the spectral sheaf of an irregular Higgs bundle $(\E, \theta)$. 
 The integrality requirement on $\Sigma_0$ was later lifted in \cite{Sch}. 

 The idea of the proof is to use the properties of proper transform
 functor of coherent sheaves under the blow-up introduced in
 \cite{ASz}.  Namely, for any smooth surface $W$ and a point $w\in W$,
 let us denote by $\tau: \widetilde{W} \to W$ the blow-up of $w$ and
 by $E$ the exceptional divisor. Now, given any coherent sheaf $\F$ of
 $\O_W$-modules we set
 $$
  \F^E:=\Tor_1^{\O_{\widetilde{W}}}(\tau^*\F, \O_{\widetilde{W}}(E)_E)
 $$ 
 and 
 $$
  \F^{\tau} = \tau^*\F / \F^E. 
 $$
 With these notations, we have the following result. 
 \begin{lem}[Lemma 5.12 \cite{ASz}] \label{lem:A-Sz}
  Suppose that the homological dimension of $\F$ at $x$ satisfies $\mbox{dh}(\F_x ) = 1$. 
\begin{enumerate}
 \item If $\F_x$ is torsion, then $\mbox{dh}(\F^{\tau}_y) = 1$ for any $y\in E$. \label{lem:A-Sz1}
 \item We have $R^0 \sigma_* (\F^{\tau}) = \F$ and $R^i \tau_* (\F^{\tau}) =0$ for all $i>0$. \label{lem:A-Sz2}
 \item If $\F$ is pure of dimension $1$ then $E \not \subseteq
   \mbox{supp}(\F^{\tau})$. \label{lem:A-Sz3} \qed
\end{enumerate}
\end{lem}

The definition of $\Ft_0$ makes it clear that it is a torsion module,
of homological dimension $1$.  As the surface $X$ is regular,
according to the Auslander--Buchsbaum formula we also get that $\Ft_0$
is pure of dimension $1$. Let us write
$$
  \Ft_1 = (\Ft_0)^{\sigma_1}.
$$
Then part (\ref{lem:A-Sz1}) of the lemma applied to $W = X$, $w\in X$ the point given by (\ref{eq:a-8_kappa}) and $\F = \Ft_0$ implies that 
$\Ft_1$ is also of homological dimension $1$, and as above we also get that it is pure of dimension $1$. 
Furthermore, part (\ref{lem:A-Sz2}) of the lemma implies that 
$$
  R^0 (\sigma_1)_* (\Ft_1 ) = \Ft_0 .
$$
We recursively define for all $ n \in \{ 2, \ldots , 8 \}$ the coherent sheaf 
$$
  \Ft_{n} = (\Ft_{n-1})^{\sigma_n} 
$$
on $X_n$. Recursive application of part (\ref{lem:A-Sz1}) of the lemma then implies that $\Ft_n$ is of homological dimension $1$ and pure of dimension $1$, 
and by part (\ref{lem:A-Sz2}) it satisfies 
$$
    R^0 (\sigma_n)_* (\Ft_n ) = \Ft_{n-1} .
$$
Let us set $\Ft = \Ft_8$. 
It then follows that using the map of (\ref{eq:sigma_twisted}) we have 
$$
  R^0 \sigma_* (\Ft ) = \Ft_0. 
$$
Using the properties of $\Ft_0$ we then get that 
\begin{equation}\label{eq:direct-image-spectral-sheaf}
  R^0 (p \circ \sigma)_* (\Ft ) = \E. 
\end{equation}
We now show that 
\begin{equation}\label{eq:functor-objects}
   \E \mapsto \Ft = \Ft_8 
\end{equation}
gives a map from the set of objects of (\ref{prop:equivalence1}) to the set of objects of (\ref{prop:equivalence2}).  
Indeed, purity follows from Lemma \ref{lem:A-Sz} as observed above. The rank of $\Ft$ is equal to $1$ because of (\ref{eq:direct-image-spectral-sheaf}), 
given that the rank of $\E$ is $2$ and that $p \circ \sigma|_{\widetilde{\Sigma}}$ is a double cover of $\CP1$. 
Finally, by part (\ref{lem:A-Sz3}) of the lemma, the exceptional divisors $E_1,\ldots ,E_{10}$ are not contained in $\widetilde{\Sigma}$. 
Moreover, according to part (\ref{lem:an-bn-1}) of Lemma \ref{lem:an-bn} and part (\ref{lem:dn-an-1}) of Lemma \ref{lem:dn-an}  for each $n$ the center of 
the blow-up $\sigma_n$ is the only intersection point of the proper transform of $\Sigma_0$ in $X_{n-1}$ with the exceptional divisor $E_{n+1}$. 
This implies the statement about intersections. 

Conversely, suppose that a sheaf $\Ft$ fulfilling the properties of
(\ref{prop:equivalence2}) is given.  Then, we define a holomorphic
vector bundle $\E$ by (\ref{eq:direct-image-spectral-sheaf}), and we
define a Higgs field $\theta$ as the direct image of multiplication by
$\zeta \d z$ on $\Ft_0 = R^0 \sigma_* (\Ft )$.  If the curve
$\widetilde{\Sigma}$ is disjoint from $E_1,\ldots ,E_9$ and intersects
$E_{10}$ with algebraic multiplicity $1$, then the expansion of its
image $\sigma(\widetilde{\Sigma} )$ near $q$ is given by
(\ref{eq:Taylor-E8}).  By virtue of part (\ref{lem:dn-an-2}) of
Lemma~\ref{lem:dn-an}, this implies the converse expansion
(\ref{eq:Puiseux-E8}) with $a_{-7} \neq 0$.  Then, according to part
(\ref{lem:an-bn-2}) of Lemma~\ref{lem:an-bn}, the coefficients in
the form (\ref{eq:local-form-E8}) are as required.  This then gives
the inverse map of (\ref{eq:functor-objects}) on objects.

Now, let us consider the map on morphisms. 
Recall that an isomorphism $(\E_1, \theta_1) \cong (\E_2, \theta_2)$ amounts to an isomorphism of vector bundles 
$$
  \Psi : \E_1 \to \E_2
$$
such that 
$$
  \theta_2 \circ \Psi = (\Psi \otimes \mbox{I}_K) \circ \theta_1, 
$$
where $\mbox{I}_K$ stands for the identity map of the canonical bundle $K_{\CP1}$. Therefore, if $(\E_1, \theta_1)$ and 
$(\E_2, \theta_2)$ are isomorphic, then we have a diagram 
$$ \xymatrixcolsep{2.1pc}\xymatrix{
0 \ar[r] 
& p^* \E_1 \ar[r]^-{\xi\otimes p^*\theta_1+\zeta} \ar[d]^{\Psi } 
& p^* (E_1 \otimes K(4 \cdot \{ 0 \})) \ar[r] \ar[d]^{\Psi \otimes \mbox{I}_K} 
& \Ft_0 (\E_1,  \theta_1) \otimes p^* (K(4 \cdot \{ 0 \})) \ar[r] & 0 \\ 
0 \ar[r] 
& p^* \E_2 \ar[r]^-{\xi\otimes p^*\theta_2+\zeta} 
& p^* (E_2 \otimes K(4 \cdot \{ 0 \})) \ar[r] 
& \Ft_0 (\E_2, \theta_2) \otimes p^* (K(4 \cdot \{ 0 \})) \ar[r] & 0}
$$

It follows from this diagram that there exists a morphism of sheaves of $\O_{\Xt}$-modules 
$$
  \Ft_0 (\E_1, \theta_1) \to \Ft_0 (\E_2, \theta_2), 
$$
which is an isomorphism with inverse induced by $\Psi^{-1}$ in the same way. 
This isomorphism in turn induces isomorphisms 
$$
  \Ft_8 (\E_1, \theta_1) \cong \Ft_8 (\E_2, \theta_2) 
$$
by functoriality of the proper transform operation. 
On the other hand, such an isomorphism of spectral sheaves gives an isomorphism of Higgs bundles by functoriality of the direct image functor. 
This finishes the proof of Proposition~\ref{prop:equivalence}. 
\end{proof}

\subsection{The proof of Theorem \ref{thm:mainE8}} 
Now we are ready to give the proof of our second main result.
\begin{proof}[Proof of Theorem \ref{thm:mainE8}]
 According to Proposition \ref{prop:equivalence}, describing the
 moduli space of irregular Higgs bundles with local form given by
 (\ref{eq:local-form-E8}) is equivalent to describing the relative
 Picard scheme of degree $1$ torsion-free sheaves on curves satisfying
 the properties listed in part (\ref{prop:equivalence2}) of 
Proposition~{\ref{prop:equivalence}}.

As in the untwisted case, we write the characteristic polynomial of 
$\t$ in the trivialization given by {$\kappa_1$ and $\kappa_1^2$} of 
(\ref{eq:char-poly2}). The polynomials $f_1$ and $g_1$ are given 
in (\ref{eq:f1}) and (\ref{eq:g1}), and the characteristic polynomial is:
\begin{equation*}
	\chi_{\vartheta_1} (z_1, w_1) = w_1 ^2- \left(p_2 z_1^2+p_1 z_1+p_0\right) w_1-\left(q_4 z_1^4+q_3 z_1^3+q_2 z_1^2+q_1 z_1+q_0\right).
\end{equation*}
Now the roots of {$\chi_{\vartheta_1}(0, w_1)$} in $w_1$ are equal, because the
curve intersects the $z_1=0$~line in one point. This requirement is satisfied
if the discriminant of {$\chi_{\vartheta_1}$} vanishes at $z_1=0$, that is,
\begin{equation*}
	p_0^2+4 q_0 = 0.
\end{equation*}

After this simplification, we consider the expansions of the roots of
{$\chi_{\vartheta_1} (z_1, w_1)$} with respect to $z_1$.  It is enough
to consider the positive root, because the  expansions of the two
roots differ in a negative sign in certain terms. The expansion is:
\begin{align*}
\begin{split}
  w_1 =& \frac{p_0}{2} + \frac{1}{2} \sqrt{2 p_0 p_1+4 q_1} \sqrt{z_1} + \frac{p_1}{2} z_1 +\frac{p_1^2+2 p_0 p_2+4 q_2}{4 \sqrt{2 p_0 p_1+4 q_1}} z_1^{3/2} +\\
	&+ \frac{p_2}{2} z_1^2 + \frac{1}{4} \sqrt{2 p_0 p_1+4 q_1} \left(\frac{ p_1 p_2+2 q_3}{ p_0 p_1+2 q_1}-\frac{\left(p_1^2+2 p_0 p_2+4 q_2\right){}^2}{16 \left( p_0 p_1+2 q_1\right){}^2}\right) z_1^{5/2}.
\end{split}
\end{align*}

We write the local form of $\t$ in the twisted case as in 
(\ref{eq:local-form-E8}). We described the matrix eigenvalues in 
Lemma~\ref{lem:an-bn} by the Puiseux expansion, with coefficients
$a_n$.

These two expansions are the same, hence by comparing the coefficients
we get the following:
\begin{align*}
\begin{split}
  \chi_{\vartheta_1} (z_1, w_1, t) =& w_1^2 - \left(b_{-4} z_1^2+b_{-6} z_1+2 b_{-8}\right) w_1 -t z_1^4-b_{-3} z_1^3+ \\
	&+ \left(b_{-8} b_{-4}-b_{-5}\right) z_1^2+ \left(b_{-8} b_{-6}-b_{-7}\right) z_1+b_{-8}^2,
\end{split}
\end{align*}
where (as in the untwisted case) we denote $q_4$ by $t$, and the
degree of the polynomial is $2$ in the variable $w_1$ and $4$ in
$z_1$. Therefore, we get a pencil parametrized by $t$ with 
base locus $(0,b_{-8})$ in $\C^2$.

As in the  untwisted case, we consider the characteristic 
polynomial in the chart $U_2$ with trivialization $\kappa_2$.
The polynomials $f_2$ and $g_2$ are given in (\ref{eq:f2}) and (\ref{eq:g2}).
\begin{align*}
	\chi_{\vartheta_2} (z_2, w_2, t) =& w_2^2 + f_2(z_2) w_2 + g_2(z_2,t) =\\
	=& w_2^2+ \left(2 b_{-8} z_2^2+b_{-6} z_2+b_{-4}\right)w_2 +b_{-8}^2 z_2^4+\left(b_{-8} b_{-6}-b_{-7}\right) z_2^3 +\\ 
	&+ \left(b_{-8} b_{-4}-b_{-5}\right) z_2^2-b_{-3} z_2-t.
\end{align*}

The pencil in the Hirzebruch surface $X$ is defined by 
$\chi_{\vartheta_1} (z_1, w_1, t)$ 
and the union of the section at infinity with fiber $F_0$. 
According to the converse direction of Theorem~\ref{thm:e8fibration}, the
pencil gives rise to an elliptic fibration in
$\CP2\# 9 {\overline {\CP{}}}^2$ with a singular fiber of type
$\Et8$.

The pencil determines the types of further singular fibers in the elliptic
fibration. In the following we will identify the types of these
further singular fibers in terms of the defining constants of the pencil.
The spectral curves intersect the fiber component $F_0$ of 
the curve $C_{\infty}$ at infinity (whose fiber is with multiplicity $4$) 
in one point and according to Condition \ref{item:SpectralCurveCond3}
of Definition~\ref{def:Hitchin_base} 
the pencil has no singular point on the distinguished fiber $F_0$.  
Thus it is sufficient to consider the $\kappa_2$ trivialization, i. e. the 
chart $(z_2, w_2)$ (see Equation~(\ref{eq:kappa2})). 
For identifying the singular fibers in 
the pencil, we look for triples $(z_2, w_2,t)$ such that $(z_2, 
w_2)$ fits the curve with parameter $t$ and the partial derivatives 
below vanish:
\begin{align*}
	\chi_{\vartheta_2} (z_2, w_2, t) &= 0, \\
	\frac{\partial \chi_{\vartheta_2}(z_2, w_2, t)}{\partial w_2} &=0, \\
	\frac{\partial \chi_{\vartheta_2}(z_2, w_2, t)}{\partial z_2} &=0.
\end{align*}

Notice that the second and third equations do not involve $t$, hence 
we can solve this system  for $w_2$ and $z_2$. Indeed, by
solving  the second equations for the variable $w_2$ we get
\begin{equation*}
	w_2 =-\frac{1}{2} \left(2 b_{-8} z_2^2+b_{-6} z_2+b_{-4}\right).
\end{equation*}
We substitute the resulting expression into the third equation, leading to
\begin{equation}
	\label{eq:E8_quadric}
	0 = 6 b_{-7} z_2^2+\left(b_{-6}^2+4 b_{-5}\right) z_2+b_{-6} b_{-4}+2 b_{-3}.
\end{equation}

This polynomial is quadratic in $z_2$ and has one root if and only if 
the discriminant 
\begin{equation*}
	D=\left(b_{-6}^2+4 b_{-5}\right){}^2-24 b_{-7} \left(b_{-6}
        b_{-4}+2 b_{-3}\right)
\end{equation*}
vanishes. In this case the pencil has a single further singular fiber, which
has a cusp singularity. If $D\neq 0$ then the fibration has two $I_1$ singular
fibers.

The fibration obtained from the pencil has a
section, so just as in the proof of Theorem~\ref{thm:mainE7} we may
apply the relative Abel--Jacobi map to identify the fibration and its
relative Picard scheme over the locus of smooth curves. Thus it is
sufficient to describe the singular fibers of $H$. By Proposition
\ref{prop:compactified-Picard-scheme-I1-II}, these latter are as
stated in Theorem~\ref{thm:mainE8}, concluding the proof.
\end{proof}

\bibliography{2d}
\bibliographystyle{plain}

\end{document}